\documentclass[a4paper, 11pt]{amsart}

\usepackage{amsfonts,mathtools,amssymb,amsthm} 

\usepackage[a4paper,margin=2.5cm]{geometry}
\usepackage{float}
\usepackage[symbol]{footmisc}
\usepackage{enumitem}

\usepackage{microtype}
\usepackage[font=small]{caption}

\usepackage[backend=bibtex,
style=numeric,
autocite=inline,maxbibnames=100,sortcites,doi=false,isbn=false,url=false]{biblatex}
\renewbibmacro{in:}{}

\usepackage{hyperref}
\hypersetup{
    colorlinks=true,
    linkcolor=blue,
    citecolor=blue,
    urlcolor=blue
}
\usepackage{graphicx,cleveref}

\usepackage{tikz-cd,quiver}
\usetikzlibrary{decorations.pathreplacing}
\usetikzlibrary{arrows}
\pgfdeclarelayer{background}
\pgfsetlayers{background,main}
\usepackage{array,booktabs}

\numberwithin{equation}{section}

\newtheorem{theorem}{Theorem}[section]
\newtheorem{proposition}[theorem]{Proposition}

\newtheorem*{conjecture*}{Conjecture}
\newtheorem*{theorem*}{Theorem}
\newtheorem{corollary}[theorem]{Corollary}
\newtheorem{lemma}[theorem]{Lemma}

\newtheorem{namedthm}{Theorem}
\newenvironment{thm*}[1]{%
  \renewcommand\thenamedthm{#1}%
  \namedthm}{\endnamedthm}

\theoremstyle{definition}

\newtheorem{example}[theorem]{Example}

\newtheorem{remark}[theorem]{Remark}

\newcommand{\C}{\mathbb{C}}

\newcommand{\Q}{\mathbb{Q}}
\newcommand{\R}{\mathbb{R}}
\newcommand{\Z}{\mathbb{Z}}

\DeclareMathOperator{\im}{Im}
\DeclareMathOperator{\PSL}{PSL}
\DeclareMathOperator{\Hom}{Hom}

\DeclareMathOperator{\Ho}{H}
\DeclareMathOperator{\coker}{Coker}
\DeclareMathOperator{\Ext}{Ext}
\DeclareMathOperator{\vcd}{vcd}
\DeclareMathOperator{\cd}{cd}
\DeclareMathOperator{\pd}{pd}
\DeclareMathOperator{\lcm}{lcm}
\DeclareMathOperator{\ord}{ord}
\DeclareMathOperator{\Isom}{Isom}
\DeclareMathOperator{\Stab}{Stab}
\DeclareMathOperator{\Ann}{Ann}
\DeclareMathOperator{\Int}{int}

\newcommand{\HH}{\mathbb{H}}
\newcommand{\EE}{\mathbb{E}}
\newcommand{\ZG}{\Z G}
\newcommand{\ZT}{\Z T}
\newcommand{\g}{\gamma}
\newcommand{\om}{\omega}
\newcommand{\lMod}[1]{#1{\operatorname{-Mod}}}
\newcommand{\catname}[1]{{\normalfont\textbf{#1}}}
\newcommand{\FP}{\catname{FP}}
\newcommand{\FL}{\catname{FL}}
\newcommand{\VFL}{\catname{VFL}}
\newcommand{\ZGmod}{\lMod{\ZG}}
\newcommand{\ZTmod}{\lMod{\ZT}}
\newcommand{\ZAmod}{\lMod{\Z A}}
\newcommand{\od}{\overline{d}}
\newcommand{\Ngk}[2]{\mathbf{N}_{#1}^{#2}}
\newcommand{\Nf}{\Ngk{f}{\phantom{l}}}
\newcommand{\Ng}{\Ngk{g}{\phantom{n}}}
\newcommand{\Nw}{\Ngk{W}{\phantom{n}}}
\newcommand{\W}{W(r,s,f,g)}
\newcommand{\wf}{W_{f}}
\newcommand{\wg}{W_{g}}
\newcommand{\ab}{\mathrm{ab}}
\newcommand{\Tab}{T^{\ab}}
\newcommand{\Gab}{G^{\ab}}
\newcommand{\orb}{\mathrm{orb}}

\title{Resolutions of orbifold generalised triangle groups}

\author{Sean P. Carroll}
    \address{Department of Mathematics, Northeastern University, Boston, Massachusetts, United States of America}
    \email{carroll.sea@northeastern.edu}

\author{Thomas Csizmadia}
    \address{Department of Mathematics, University of Auckland, Auckland, New Zealand}
    \email{tcsi800@aucklanduni.ac.nz}

\author{Jeroen Schillewaert}
    \address{Department of Mathematics, University of Auckland, Auckland, New Zealand}
    \email{j.schillewaert@auckland.ac.nz}

\begin{document}
    \makeatletter
    \@namedef{subjclassname@2020}{\textup{2020} Mathematics Subject Classification}
    \makeatother

    \begin{abstract} 
        Let $T=\langle f,g\mid f^{l}=g^{m}=W^{n}=1\rangle$ be a generalised triangle group in which $W$ is the Farey word of a rational slope $q=r/s$, and suppose that $T$ is the fundamental group of a hyperbolic $3$-orbifold whose singular locus has no spherical vertices. Let $G=\langle f,g\mid f^{l}=g^{m}=W^{n}\rangle$ be the associated central extension. We compute an explicit free resolution of the trivial module over $\mathbb{Z}G$ of minimal length, with differentials given in closed form by Fox derivatives. From this resolution we deduce an explicit free resolution over $\mathbb{Z}T$ which is $2$-periodic in degrees at least two, and we compute the integral homology and cohomology of both groups. Strebel treated the classical family $q=0/1$ in 1983. Since then only the family $q=1/2$ has been treated, by Ellis and Williams, and for that family our resolution over $\mathbb{Z}T$ has strictly smaller rank in every degree.
    \end{abstract} 

    \maketitle
    \vspace{-.25in}

  
    \section{Introduction}

    The history of triangle groups is, in miniature, the history of the idea that a group can be understood through the space on which it acts. Poincar\'e \cite{Poincare1882}, Klein and Fricke \cite{Klein1884,KleinFricke1897} transformed the elementary relations of a triangle into a theory of symmetries, tessellations and automorphic functions. In modern language the triangle becomes a two-dimensional orbifold, and its universal cover is the geometric object from which the group's structure can be read off. It is natural to investigate how this picture extends when the two-dimensional orbifold is replaced by its three-dimensional analogue: the quotient $\mathbb{H}^3/\Gamma$ by a Kleinian group $\Gamma$ generated by two elliptic elements, which again carries a singular locus governed by the orders of the generators and their product. In the parabolic case, Agol conjectured \cite{Agol2002} that every non-free Kleinian group generated by two non-commuting parabolic elements is either a hyperbolic $2$-bridge link group or a Heckoid group, and this was proved by Akiyoshi, Ohshika, Parker, Sakuma and Yoshida \cite{AOPSY}. The elliptic case is expected to exhibit the same dichotomy, with generalised triangle groups and elliptic Heckoid groups as its two classes. It is the former class, the one most directly descended from the triangle groups of Klein and Poincar\'e, that occupies us here.

    An \emph{orbifold generalised triangle group} is a group with a finite presentation of the form
    \[T=T_{W}(l,m,n) = \langle f,g \mid f^{l}=g^{m}= W^{n}=1\rangle,\]
    where $W$ is a \emph{Farey word}
    \begin{equation}\label{eq:Wrs}
        \W=\prod_{i=0}^{s-1}f^{a_{i}}g^{b_{i}},
        \quad
        a_{i}=(-1)^{\lfloor(2i+1)r/s\rfloor},
        \quad
        b_{i}=(-1)^{\lfloor(2i+2)r/s\rfloor},
    \end{equation}
    indexed by a rational slope $q\coloneqq r/s$ for coprime integers $r,s$ with $0\le r<s$. Jones and Reid proved \cite[Thm.~8]{jonesreid} that $T$ is the orbifold fundamental group of a compact orientable $3$-orbifold $Q_{0}$, possibly with boundary. This orbifold is obtained by truncating a space $Q_{r,s}(l,m,n)$ whose singular locus is a $2$-bridge knot or link together with its unknotting tunnel.
    
    If we drop the requirement that the three powers be trivial, we obtain the group
    \[G=G_{W}(l,m,n) = \langle f,g \mid f^{l}=g^{m}= W^{n}\rangle, \qquad c\coloneqq f^{l}=g^{m}=W^n\in G,\]
    in which $c$ is central and $1\to\langle c\rangle\to G\to T\to 1$ is a central extension. Let
    \begin{equation*}
        c_{1}=\begin{cases}
            \frac{1}{l}+\frac{1}{m}+\frac{1}{n}-1 & s \text{ odd},\\
            \frac{2}{l}+\frac{1}{n}-1 & s \text{ even},
        \end{cases}
        \quad
        c_{2}=\begin{cases}
            \frac{1}{l}+\frac{1}{m}+\frac{1}{n}-1 & s \text{ odd},\\
            \frac{2}{m}+\frac{1}{n}-1 & s \text{ even}
        \end{cases}
    \end{equation*} 
    be the invariants of Williams \cite[\S2]{willectri}. Despite the notation, these invariants are unrelated to the central element $c\in Z(G)$. As explained in \S\ref{subsec:turnovers}, $c_{i}$ is the orbifold Euler characteristic of the link of the $i$-th vertex of the singular locus, so $c_{i}\le0$ means that this link is Euclidean or hyperbolic. Our paper concerns orbifold generalised triangle groups and their central extensions when both vertex links are non-spherical, that is, when $\max(c_{1},c_{2})\le 0$.
      
    \begin{thm*}{A}\label{thm:summaryA}
        Suppose $\max(c_1,c_2)\leq 0$ and $s\geq 2$. Then
        \begin{enumerate}[label=\textup{(\alph*)}]
            \item $\Z\in\ZGmod$ has a free resolution of length three with differentials given in closed form by Fox derivatives \textup{(Theorem \ref{thm:Gres})}.
    
            \item $\Z\in\ZTmod$ has a free resolution which is $2$-periodic in degrees at least two, with stable rank 3 \textup{(Theorem \ref{thm:Tres})}.
        \end{enumerate}
    \end{thm*}
    For the group $G$ a classical theorem of Lyndon (Theorem \ref{Thm:preres}) provides the first two differentials, and we determine the remaining kernel explicitly. The resolution over $\ZT$ is then obtained by the construction of \cite{seanyc}. Both steps rest on properties of $G$ and $T$ that we first derive from the geometry of the orbifold, and the derivation of these properties occupies most of the paper.

    Once the length-three resolution over $\ZG$ is available, the integral homology and cohomology are easy to compute. We apply the augmentation map to the entries of the differentials and compute the Smith normal forms of the resulting integer matrices.
    
    \begin{thm*}{B}\label{thm:homologyG}
        Suppose $\max(c_1,c_2)\leq 0$ and $s\geq 2$. Then
        \[
          \Ho_{i}(G,\Z)\cong\begin{cases}
              \Z & i=0,\\
              \Gab & i=1,\\
              0 & i=2 \textup{ and } \det J\ne 0,\\
              \Z & i=2 \textup{ and } \det J=0,\\
              \Z & i=3,\\
              0 & i\ge 4,
          \end{cases}
          \qquad
          \Ho^{i}(G,\Z)\cong\begin{cases}
              \Z & i=0,\\
              0 & i=1 \textup{ and } \det J\ne 0,\\
              \Z & i=1 \textup{ and } \det J=0,\\
              \Gab & i=2 \textup{ and } \det J\ne 0,\\
              \Z\oplus\Z/\nu & i=2 \textup{ and } \det J=0,\\
              \Z & i=3,\\
              0 & i\ge 4.
          \end{cases}
        \]
    \end{thm*}
    Here $\nu$ is the integer defined in \eqref{eq:nu-g} and $J$ is the matrix \eqref{eq:J}. 

    For the orbifold generalised triangle groups themselves we obtain the following computation. This computation is the exact analogue of the classical result for the triangle groups, which form the case $W=fg$ (see e.g.\ \cite[Prop.~6.4.1]{ellis}).
    \begin{thm*}{C}\label{thm:homologyT}
        Suppose $\max(c_1,c_2)\leq 0$ and $s\geq 2$. Then
        \[
          \Ho_{i}(T,\Z)\cong\begin{cases}
              \Z & i=0,\\
              \Tab & i=1,\\
              \Z & i=2,\\
              \Z/l\oplus\Z/m\oplus\Z/n & i\ge 3, \ i\ \textup{odd},\\
              0 & i\ge 4,\ i\ \textup{even},
          \end{cases}
          \Ho^{i}(T,\Z)\cong\begin{cases}
              \Z & i=0,\\
              0 & i=1,\\
              \Z\oplus\Tab & i=2,\\
              0 & i\ge 3, \ i\ \textup{odd},\\
              \Z/l\oplus\Z/m\oplus\Z/n & i\ge 4,\ i\ \textup{even}.
          \end{cases}
        \]
    \end{thm*}
    The abelianisations $\Tab$ and $\Gab$ are computed in Propositions \ref{prop:Tab} and \ref{prop:Gab}.

    \medskip
   
    \textbf{Related work.} Orbifold generalised triangle groups form an important class of generalised triangle groups, which have a similar presentation but allow $W=W(f,g)\in\Z/l\ast\Z/m$ to be any cyclically reduced word of syllable length at least $2$, with $l,m,n\ge 2$ integers. Some authors also require $W$ not to be conjugate to a proper power \cite[\S1]{jonesreid}. Generalised triangle groups were introduced by Baumslag, Morgan and Shalen \cite{BaumslagMorganShalen}, and much of the subsequent work has been combinatorial, resting on Howie's analysis of one-relator products with a high-powered relator \cite{howiepictures,howiefourth}. The finite ones are classified completely \cite{LevaiRosenbergerSouvignier1995,HowieMetaftsisThomas1995}, and they are among the cases excluded in the proof of Lemma \ref{ass}. The family of Farey words \eqref{eq:Wrs}, by contrast, is geometric in origin. Edmonds, Ewing and Kulkarni \cite{EdmondsEwingKulkarni1982} showed that for a finitely generated Fuchsian group of the first kind, the minimal index of a torsion-free subgroup is bounded above by twice the least common multiple of the orders of the finite (cyclic) subgroups. Jones and Reid introduced the orbifold generalised triangle groups to construct Kleinian groups for which this minimal index exceeds the least common multiple of the orders of the finite subgroups by an arbitrarily large factor \cite[Thm.~1]{jonesreid}. In the process Jones and Reid answered in part a question of Hagelberg on which generalised triangle groups arise as orbifold fundamental groups \cite[Thm.~8]{jonesreid}. Williams then determined which of the resulting Kleinian groups are arithmetic \cite{willarith}. In the non-cocompact case this sits inside Maclachlan and Martin's determination of all non-cocompact arithmetic Kleinian groups with two elliptic generators \cite{maclachlanmartin}. There are twenty-one such groups and fifteen of them are orbifold generalised triangle groups.

    Resolutions for generalised triangle groups have a shorter history. Strebel \cite{streb} treated the classical word $W=fg$, the case $q=0/1$ of the present setting, giving explicit free resolutions for the binary polyhedral groups and, in the infinite cases, for the central extensions of the triangle groups. Later the first author computed eventually $2$-periodic free resolutions for the hyperbolic triangle groups themselves \cite[\S10.3]{seanyc}. Ellis and Williams \cite{ellistri} gave a general method for constructing free resolutions for groups which act suitably on a contractible space, applied it to the generalised triangle groups with $W(1,2,f,g)=fg^{-1}f^{-1}g$, and computed their integral homology. The method of \cite{ellistri} also recovers the resolutions for the triangle groups and for the infinite cyclic central extensions treated by Strebel. 

    \medskip
    
    \textbf{Outline and strategy of the paper.} 
    Section \ref{sec:setup} fixes notation. Section \ref{sec:orbifold} collects in one place everything we use about the space $Q_{r,s}(l,m,n)$ and the associated truncated orbifold $Q_{0}$. In particular, Theorem \ref{thm:geometry} records the geometric structure of $Q_{0}$. The condition $\max(c_{1},c_{2})\le0$ says that both vertices of the singular locus of $Q_{r,s}(l,m,n)$ have non-spherical links, so that both are truncated, and this is what makes the whole approach possible. Indeed, when $\max(c_{1},c_{2})>0$ and the truncated orbifold is hyperbolic, or when $T$ is finite, the central subgroup $\langle c\rangle$ is finite. As a result, $c-1$ is a zero-divisor in $\ZG$ and $G$ has torsion, and neither the resolution of Theorem \ref{thm:Gres} nor the periodic resolution of Theorem \ref{thm:Tres} is available (Remark \ref{rem:badcases}).
    
    Sections \ref{sec:vcdT} to \ref{sec:cdG} establish the properties of $T$ and $G$ on which the resolutions rest. In Section \ref{sec:vcdT} we compute the virtual cohomological dimension of $T$. The abelianisations of $G$ and $T$ are computed in Section \ref{sec:ab}, and the second rational homology of $T$ in Section \ref{sec:H2}. Section \ref{sec:central} combines these to show that $c$ has infinite order in $G$, and it also computes $\vcd(G)$. Section \ref{sec:cdG} locates the torsion of $T$ and shows that $G$ is torsion-free of type $\FP$ with $\cd(G)=3$. The resolutions of Theorem \ref{thm:summaryA} occupy Section \ref{sec:freeres}. The resolution over $\ZG$ uses from the earlier sections only Theorem \ref{thm:geometry}, the fact that $c\in G$ has infinite order (Theorem \ref{thm:central}) and the torsion of $T$ (Lemma \ref{lem:torsion-class}). The key step is Lemma \ref{hyp:AS}, which identifies the module of identities among relations. Rather than appeal to picture theory we read the answer off the universal orbifold cover of $Q_{0}$ and translate it into Fox calculus. The homology and cohomology computations of Theorems \ref{thm:homologyG} and \ref{thm:homologyT} are in Section \ref{sec:homology}, and two worked examples are given in \S\ref{sec:examples}.

    \medskip
    
    \textbf{AI usage.} A completed draft of the paper was obtained by the authors over the past two years without any use of AI. We used OpenAI GPT-6 Astra for proofreading which picked up minor inaccuracies.
    
    \section{Notation and conventions}\label{sec:setup}

    \subsection{The groups}\label{subsec:groups}
    Throughout, $l,m,n\ge 2$ and $r,s$ are coprime integers with $0\le r<s$. With $\W$ as in \eqref{eq:Wrs} we put
    \[
        T=\langle f,g\mid f^{l}=g^{m}=\W^n=1\rangle,
        \quad
        G=\langle f,g\mid f^{l}=g^{m}=\W^n\rangle,
        \quad
        c\coloneqq f^{l},
    \]
    so that $c\in Z(G)$ and $1\to\langle c\rangle\to G\to T\to 1$ is a central extension. Let $\sigma_{f},\sigma_{g}$ denote the exponent sums of $f,g$ in $\W$. A direct computation from \eqref{eq:Wrs} gives
    \[
        \sigma_{f}=\sigma_{g}=0 \ \text{ for } s \text{ even},
        \quad
        \sigma_{f},\sigma_{g}\in\{\pm 1\} \ \text{ for } s \text{ odd}.
    \]
    The two signs may differ. For instance $r=1$ and $s=3$ give $W(1,3,f,g)=fgf^{-1}g^{-1}f^{-1}g$ with $\sigma_{f}=-1$ and $\sigma_{g}=+1$. We also set
    \begin{equation}\label{eq:nu-g}
        \eta\coloneqq \gcd(l,m),
        \quad
        \nu\coloneqq \gcd(l,m,n\sigma_{f},n\sigma_{g}) =\begin{cases}
            \gcd(l,m)=\eta & s \text{ even},\\
            \gcd(l,m,n) & s \text{ odd}.
    \end{cases}
    \end{equation}
    By Propositions \ref{prop:Tab} and \ref{prop:Gab}, the integer $\nu$ is the first invariant factor of both $\Tab$ and $\Gab$.
    
    \subsection{Orbifolds}\label{subsec:orbifolds}
    We use the standard language of $3$-orbifolds, for which \cite{BoileauMaillotPorti2003, CHK} are convenient references. An orbifold is a space locally modelled on quotients $\widetilde U/K$ of open subsets of $\R^{d}$, or of a half-space where boundary is allowed, by finite groups acting smoothly and faithfully. Each point $x$ has a \emph{local group} $\Gamma_{x}$, well defined up to isomorphism, and the \emph{singular locus} is the set of $x$ with $\Gamma_{x}\ne 1$. Orbifold coverings, the universal orbifold covering and $\pi_1^{\orb}$ as its deck group are as in \cite[\S2.2]{BoileauMaillotPorti2003}. In particular, if $p\colon E\to B$ is an orbifold covering with $B$ connected and $E_{0}$ is a component of $E$, then $p|_{E_{0}}\colon E_{0}\to B$ is again an orbifold covering since its image is both open and closed. Further, if in addition $\pi_{1}^{\orb}(E_{0})=1$, for example if $E_{0}$ is a simply connected manifold, then $p|_{E_{0}}$ is the universal orbifold covering of $B$, with deck group $\pi_{1}^{\orb}(B)$ (cf.\ \cite[Thm.~2.6]{BoileauMaillotPorti2003}).

    If a group $\Gamma$ acts properly discontinuously and faithfully by diffeomorphisms on a manifold $M$, then $M/\Gamma$ is an orbifold whose local group at the image of a point $y$ is the stabiliser $\Gamma_{y}$, which is finite by properness. An orbifold of this form is \emph{good} \cite[\S2.2]{BoileauMaillotPorti2003}. If $M$ is in addition simply connected, then $M\to M/\Gamma$ is the universal orbifold covering, so that $\pi_1^{\orb}(M/\Gamma)=\Gamma$.
    
    A closed $2$-orbifold with underlying surface $S$ and cone points of orders $\alpha_{1},\dots,\alpha_t$ has \emph{orbifold Euler characteristic}
    \[
      \chi^{\orb}=\chi(S)-\sum_{j=1}^{t}\Bigl(1-\frac1{\alpha_{j}}\Bigr),
    \]
    and, provided it is not one of the bad orbifolds $S^{2}(\alpha)$ and $S^{2}(\alpha,\beta)$ with $\alpha\ne\beta$, it carries a spherical, Euclidean or hyperbolic structure according as $\chi^{\orb}$ is positive, zero or negative.

    \subsection{Group finiteness properties}\label{subsec:finiteness}
    For a group $A$ we write $\Z A$ for the integral group ring, $\varepsilon\colon\Z A\to\Z$ for the augmentation map, and we set $\cd(A)\coloneqq\pd_{\Z A}(\Z)$. The standard reference for what follows is \cite[Ch.~VIII]{brown}.
     
    We say that $A$ is \emph{of type $\FP$} if the trivial module $\Z\in\ZAmod$ admits a projective resolution that is finitely generated in all degrees and of finite length. Similarly, we say $A$ is \emph{of type $\FL$} if $\Z\in\ZAmod$ admits a \emph{free} resolution that is finitely generated in all degrees and of finite length. When $\cd(A)<\infty$ that length may be taken to be $\cd(A)$ (cf.\ \cite[Ex.~1, p.~203]{brown}). Since a group with torsion has infinite cohomological dimension \cite[Ch.~VIII, Cor.~2.5]{brown}, it follows that any group of type $\FP$ is necessarily torsion-free.
     
    A virtually torsion-free group $\Gamma$ is \emph{of type $\VFL$} if some torsion-free subgroup of finite index is of type $\FL$, and we set $\vcd(\Gamma)=\cd(\Gamma')$ for such a subgroup $\Gamma'$. This is independent of the choice of $\Gamma'$ by a theorem of Serre \cite[Prop.~5]{serredisc}.

    \section{The truncated orbifold}\label{sec:orbifold}
    This section assembles the relevant information about the space $Q\coloneqq Q_{r,s}(l,m,n)$ and its truncation $Q_{0}$. The singular graph of $Q$ and the turnovers arising as vertex links are described in \S\S\ref{subsec:graph}--\ref{subsec:turnovers}. In \S\ref{subsec:structures} the space is truncated at those vertices whose link is not spherical, and Theorem \ref{thm:geometry} records the position of the truncation $Q_{0}$ inside the associated hyperbolic quotient together with the structure of its universal orbifold covering. The peripheral subgroups carried by the boundary turnovers are identified in \S\ref{subsec:peripheral}.
    
    \subsection{The singular graph}\label{subsec:graph}
    The underlying space of $Q$ is $S^{3}$ and its singular locus is a $2$-bridge knot or link $L$ together with its unknotting tunnel $\tau$ \cite[\S1]{jonesreid}, \cite[\S1]{willectri}. The addition of the tunnel edge to $L$ produces a trivalent graph $L \cup \tau$ of one of the two combinatorial types in Figure \ref{fig:singular-graphs}.
    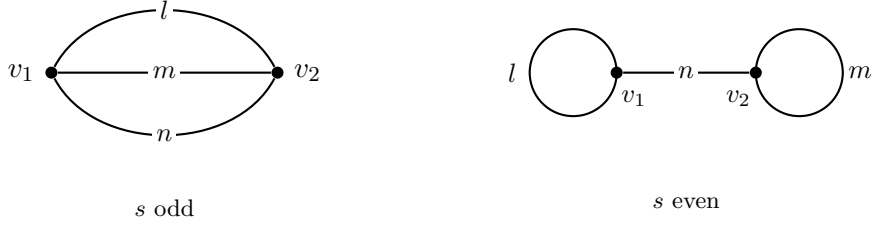
\begin{figure}[ht]
        \centering
        \begin{tikzpicture}[scale=1.15,
            vtx/.style={circle,fill=black,inner sep=1.6pt},
            lbl/.style={fill=white,inner sep=1.5pt,font=\small}]
    
            \begin{scope}[shift={(-3.1,0)}]
                \node[vtx,label=left:$v_{1}$] (v1) at (-1.3,0) {};
                \node[vtx,label=right:$v_{2}$] (v2) at (1.3,0) {};
                \draw[thick] (v1) to[out=62,in=118] node[lbl,pos=0.5]{$l$} (v2);
                \draw[thick] (v1) to[out=0,in=180]  node[lbl,pos=0.5]{$m$} (v2);
                \draw[thick] (v1) to[out=-62,in=-118] node[lbl,pos=0.5]{$n$} (v2);
                \node[font=\footnotesize] at (0,-1.55) {$s$ odd};
            \end{scope}

            \begin{scope}[shift={(2.9,0)}]
                \node[vtx] (w1) at (-0.8,0) {};
                \node[vtx] (w2) at (0.8,0) {};
                \node[font=\small] at (-0.60,-0.3) {$v_{1}$};
                \node[font=\small] at (0.60,-0.3) {$v_{2}$};
                \draw[thick] (-1.3,0) circle (0.5);
                \node[lbl] at (-2,0) {$l$};
                \draw[thick] (1.3,0) circle (0.5);
                \node[lbl] at (2,0) {$m$};
                \draw[thick] (w1) -- node[lbl]{$n$} (w2);
                \node[font=\footnotesize] at (0,-1.5) {$s$ even};
              \end{scope}
        \end{tikzpicture}
        \caption{The singular graph $L \cup \tau$. The tunnel $\tau$ is the $n$-labelled edge.}
        \label{fig:singular-graphs}
    \end{figure}
    
    In the construction of Jones and Reid \cite[proof of Thm.~8]{jonesreid}, singular $2$-handles are attached to the boundary of a genus-two handlebody $H$ along disjoint simple closed curves representing $f$, $g$ and $\W$ (see \S\ref{subsec:pictures}). Then the link $L$ is a two-component link when the curve representing $\W$ separates $\partial H$, and a knot otherwise. In the pillowcase model of the four-times punctured sphere \cite[Fig.~7]{jonesreid}, this curve has slope $q=r/s$, and the corners of the two vertical edges come from the $f$-curve and the $g$-curve respectively. Since the endpoints of an arc of slope $q$ differ by $(s,r)$ modulo $2$, the curve separates the $f$-corners from the $g$-corners, and hence separates $\partial H$, precisely when $s$ is even. The parity of $s$ therefore decides between the two cases:
    \begin{itemize}
        \item \emph{$s$ odd ($L$ a knot).} The two tunnel endpoints cut the knot into two arcs $K_{1}, K_{2}$ and, together with $\tau$, these are the three edges of the singular graph. Combinatorially $L \cup \tau$ is a graph with two vertices joined by three edges with labels $l, m$ and $n$ (with $K_{1}, K_{2}$ carrying $l,m$ and $\tau$ carrying $n$).
        
        \item \emph{$s$ even ($L$ a link).} The tunnel joins the two components $L_{1}, L_{2}$, each of which gives a vertex of the singular graph and contributes a loop edge. Combinatorially $L \cup \tau$ is a graph with an $l$-labelled loop ($L_{1}$) at $v_{1}$ and an $m$-labelled loop ($L_{2}$) at $v_{2}$, joined by the tunnel edge $\tau$ of label $n$.
    \end{itemize}
    
    In both cases the three edge labels are $l$, $m$ and $n$, and the local groups along the edges are the cyclic groups $\Z/l$, $\Z/m$ and $\Z/n$, generated up to conjugacy by $f$, $g$ and $\W$ respectively.

    \subsection{Turnovers and the invariants \texorpdfstring{$c_{1},c_{2}$}{c1, c2}}\label{subsec:turnovers}
    The link of a trivalent vertex with incident labels $(\alpha,\beta,\gamma)$ is the \emph{turnover} $S^{2}(\alpha,\beta,\gamma)$, a $2$-sphere with three cone points of those orders, whose orbifold fundamental group is the orientation-preserving triangle group
    \[
      \Delta(\alpha,\beta,\gamma) =\langle x,y,z\mid x^\alpha=y^\beta=z^\gamma=xyz=1\rangle .
    \]
    This group is the local group at the vertex when it is finite, and only then, as explained in \S\ref{subsec:structures}.
    For a turnover the formula for $\chi^{\orb}$ in \S\ref{subsec:orbifolds} gives 
    \[
        \chi^{\orb}\bigl(S^{2}(\alpha,\beta,\gamma)\bigr)=-1+\frac{1}{\alpha}+\frac{1}{\beta}+\frac{1}{\gamma}, 
    \]
    so the turnover is the double of a spherical, Euclidean or hyperbolic triangle with angles $\pi/\alpha,\pi/\beta,\pi/\gamma$ according to the sign of $\chi^{\orb}$, as in Figure \ref{fig:turnover}. On the two vertex links of Figure \ref{fig:singular-graphs} this formula gives the Williams invariants \cite[\S2]{willectri}
    \begin{equation*}\label{eq:c1c2}
        c_{1}=\begin{cases}
            \frac{1}{l}+\frac{1}{m}+\frac{1}{n}-1 & s \text{ odd},\\
            \frac{2}{l}+\frac{1}{n}-1 & s \text{ even},
        \end{cases}
        \quad
        c_{2}=\begin{cases}
            \frac{1}{l}+\frac{1}{m}+\frac{1}{n}-1 & s \text{ odd},\\
            \frac{2}{m}+\frac{1}{n}-1 & s \text{ even},
        \end{cases}
    \end{equation*}
    and we define the \emph{end count}
    \begin{equation*}\label{eq:end-count}
      k\coloneqq\#\{i\in\{1,2\}:c_{i}\le 0\}\in\{0,1,2\} ,
    \end{equation*}
    so that $k=2$ if and only if $\max(c_{1},c_{2})\le0$. For $s$ odd we have $c_{1}=c_{2}$ and hence $k\in\{0,2\}$. The value $k=1$ therefore occurs only for even $s$. We write $F_{i}$ for the $i$-th vertex link, so that $\chi^{\orb}(F_{i})=c_{i}$.
    
    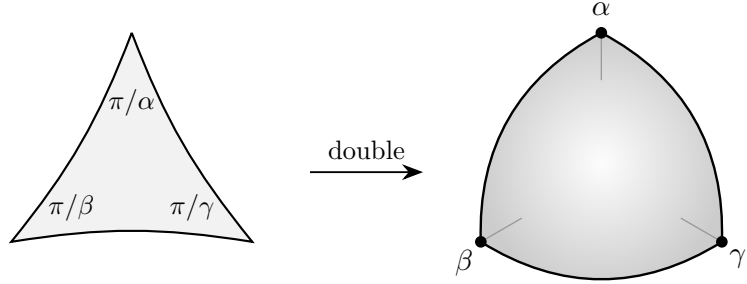
\begin{figure}[ht]
        \centering
        \begin{tikzpicture}[scale=1.15]
            \def\R{1.6}
            \begin{scope}[shift={(-3.9,0)}]
                \coordinate (O) at (0,0);
                \coordinate (a) at (90:\R); \coordinate (b) at (210:\R);
                \coordinate (c) at (330:\R);
                \path[fill=black!5, draw=black, line width=0.8pt]
                  (a) to[bend left=9] (b) to[bend left=9] (c) to[bend left=9] (a);
                \node[font=\small] at ($(a)!0.50!(O)$) {$\pi/\alpha$};
                \node[font=\small] at ($(b)!0.50!(O)$) {$\pi/\beta$};
                \node[font=\small] at ($(c)!0.50!(O)$) {$\pi/\gamma$};
            \end{scope}
            \draw[-{Stealth[length=2.8mm]}, line width=0.7pt] (-1.85,0) -- node[above=1pt, font=\small] {double} (-0.55,0);
            \begin{scope}[shift={(1.5,0)}]
                \coordinate (O) at (0,0);
                \coordinate (A) at (90:\R); \coordinate (B) at (210:\R);
                \coordinate (C) at (330:\R);
                \shade[inner color=black!1, outer color=black!20, draw=black,
                       line width=0.9pt]
                  (A) to[bend right=32] (B) to[bend right=32] (C) to[bend right=32] (A);
                \foreach \P in {A,B,C} \draw[black!40, line width=0.5pt]
                  (\P) -- ($(\P)!0.34!(O)$);
                \fill (A) circle (1.9pt) node[above=3pt] {$\alpha$};
                \fill (B) circle (1.9pt) node[below left=-1pt] {$\beta$};
                \fill (C) circle (1.9pt) node[below right=-1pt] {$\gamma$};
            \end{scope}
        \end{tikzpicture}
        \caption{The turnover $S^{2}(\alpha,\beta,\gamma)$ as the double of a triangle with angles $\pi/\alpha,\pi/\beta,\pi/\gamma$ (cf.\ \cite[Ex.~3.2]{CHK}).}
        \label{fig:turnover}
    \end{figure}

    The possible Euclidean turnovers can be listed explicitly.
    
    \begin{lemma}\label{lem:euclidean-links}
        Suppose $c_{i}=0$. Then the associated turnover $F_{i}$ is one of $S^{2}(3,3,3)$, $S^{2}(2,4,4)$ and $S^{2}(2,3,6)$, and $\pi_{1}^{\orb}(F_{i})$ is the crystallographic group $p3$, $p4$ or $p6$ respectively. In particular, if $s$ is odd then $c_{1}=0=c_{2}$ and $(l,m,n)$ is a permutation of $(3,3,3)$, $(2,4,4)$ or $(2,3,6)$. When $s$ is even, $(l,n)\in\{(4,2),(3,3)\}$ if $c_{1}=0$ and $(m,n)\in\{(4,2),(3,3)\}$ if $c_{2}=0$.
    \end{lemma}
    
    \begin{proof}
        The condition $c_{i}=0$ reads $1/\alpha+1/\beta+1/\gamma=1$ with integers $\alpha,\beta,\gamma\ge 2$. The only solutions are the three listed triples, and the associated triangle groups are the crystallographic groups $p3$, $p4$ and $p6$ \cite[\S7.5]{Ratcliffe2006}. For $s$ odd the link is $S^{2}(l,m,n)$. 
        
        Now suppose $s$ is even and $c_{1}=0$. Then the link is $S^{2}(l,l,n)$, so $2/l+1/n=1$ and $l=2n/(n-1)$, which is an integer at least $2$ only for $(l,n)=(4,2)$ and $(l,n)=(3,3)$. The same applies to $S^{2}(m,m,n)$ and $(m,n)$ if $c_{2}=0$.
    \end{proof}

    \subsection{Truncation and the hyperbolic structure}\label{subsec:structures}
    The cone on a closed $2$-orbifold $F$ carries the structure of an orientable orbifold only if $F\cong S^{2}/K$ for a finite $K\le \operatorname{SO}(3)$, that is, only if $F$ is spherical. Since $\chi^{\orb}(F_{i})=c_{i}$, the vertex $v_{i}$ is a cone point of $Q$ precisely when $c_{i}>0$. Every other vertex admits no orbifold chart, and Jones and Reid delete an open neighbourhood of it \cite[\S1]{jonesreid}  (see Figure \ref{fig:truncation}). We take $B(v_{i})$ to be an open ball around $v_{i}$ meeting $L\cup\tau$ in three radii, so that its labelled boundary sphere is the link $F_{i}$. Following \cite[proof of Thm.~2.1]{willectri} we write
    \[
        Q_0\coloneqq Q\setminus\bigsqcup_{\substack{i\\ c_{i}\le 0}} B(v_{i}),\quad \partial Q_{0} = \bigsqcup_{\substack{i\\ c_{i}\le 0}} \partial(B(v_{i})),
    \]
    for the resulting \emph{truncated} orbifold\footnote{This orbifold is denoted $Q$ in \cite[Thm.~9]{jonesreid}.}. It is compact and orientable with $\pi_{1}^{\orb}(Q_{0})\cong T$, and its boundary is the disjoint union of the $k$ turnovers $F_{i}$ with $c_{i}\le 0$, which are Euclidean for $c_{i}=0$ and hyperbolic for $c_{i}<0$. Its singular locus $(L\cup\tau)\cap Q_{0}$ has $2-k$ vertices and meets each boundary turnover transversely in that turnover's three cone points, so the local groups along the edges are those of \S\ref{subsec:graph}.
    
    \begin{figure}[ht]
        \centering
        \begin{tikzpicture}[scale=0.98, 
            vtx/.style={circle,fill=black,inner sep=1.7pt}, 
            lbl/.style={fill=white,inner sep=1.4pt,font=\small}]
            \begin{scope}[shift={(-3.6,0)}]
                \node[vtx] (c) at (0,0) {};
                \draw[thick] (c) -- ++(90:1.45) node[lbl,pos=0.62]{$\alpha$};
                \draw[thick] (c) -- ++(210:1.45) node[lbl,pos=0.62]{$\beta$};
                \draw[thick] (c) -- ++(330:1.45) node[lbl,pos=0.62]{$\gamma$};
                \draw[blue!55!black,thick,dash pattern=on 2pt off 1.5pt] (0,0) circle (0.62);
                \node[blue!55!black,font=\scriptsize] at (0.50,0.85) {link};
            \end{scope}
            \node[font=\footnotesize,align=center] at (-3.6,-1.5) {$\dfrac{1}{\alpha}+\dfrac{1}{\beta}+\dfrac{1}{\gamma}>1$};
            
            \begin{scope}[shift={(3.6,0)}]
                \fill[black!7] (0,0) circle (0.62);
                \draw[thick,red!60!black,dash pattern=on 2pt off 1.5pt] (0,0) circle (0.62);
                \draw[thick] (90:0.62) -- ++(90:0.83) node[lbl,pos=0.6]{$\alpha$};
                \draw[thick] (210:0.62) -- ++(210:0.83) node[lbl,pos=0.6]{$\beta$};
                \draw[thick] (330:0.62) -- ++(330:0.83) node[lbl,pos=0.6]{$\gamma$};
                \node[font=\normalsize] at (0,0) {\textcolor{red!60!black}{$\varnothing$}};
                \node[red!60!black,font=\scriptsize] at (0.85,0.85) {boundary};
            \end{scope}
            \node[font=\footnotesize,align=center] at (3.6,-1.5) {$\dfrac{1}{\alpha}+\dfrac{1}{\beta}+\dfrac{1}{\gamma}\le 1$};
        \end{tikzpicture}
        \caption{Left: a spherical vertex link bounds an orbifold chart, so the vertex is a cone point. Right: a Euclidean or hyperbolic vertex link admits no such chart, so an open neighbourhood of the vertex is deleted and the turnover becomes a boundary component.}
        \label{fig:truncation}
    \end{figure}
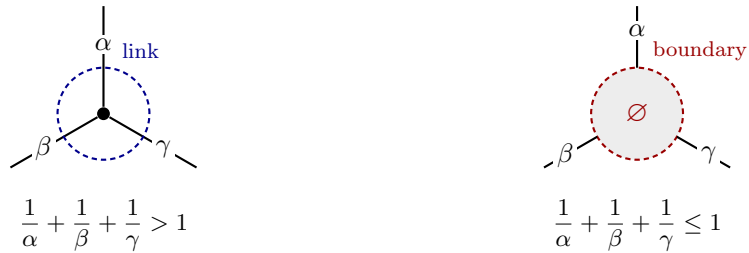

    If we delete the vertices themselves, rather than neighbourhoods of them, we obtain the orbifold without boundary
    \[
        Q^{\times}\coloneqq Q\setminus\{v_{i}\ :\ c_{i}\le 0\},
        \qquad
        U_{i}\coloneqq\overline{B(v_{i})}\setminus\{v_{i}\},
        \qquad
        Q^{\times}=Q_{0}\cup\bigsqcup_{\substack{i\\ c_{i}\le 0}}U_{i},
    \]
    in which $Q_{0}$ and the $U_{i}$ meet along the boundary $\partial Q_{0}$. Each $U_{i}$ is a cone on $F_{i}$ with its apex deleted, so
    \[
        U_{i}\cong F_{i}\times[0,\infty)
    \]
    by a diffeomorphism which maps $F_{i}$ to $F_{i}\times\{0\}$. Thus $Q^{\times}\cong\Int(Q_{0})$, and in the language of \cite[\S3]{GarciaPorti} the orbifold $Q^{\times}$ has one end $U_{i}$ with turnover cross-section $F_{i}$ for each $i$ with $c_{i}\le 0$.
    
    Recall that the universal orbifold cover $\widetilde{F_{i}}$ of $F_{i}$ is either $\EE^{2}$ or $\HH^{2}$ depending on whether $F_{i}$ is Euclidean or hyperbolic (\S\ref{subsec:turnovers} and \cite[\S7.2]{Ratcliffe2006}). Then the group $\pi_{1}^{\orb}(F_{i})$ acts on $\widetilde{F_{i}}\times[0,\infty)$ through the first factor, and this action is properly discontinuous and faithful with quotient $U_{i}$. By \S\ref{subsec:orbifolds} the inclusion $F_{i}\subseteq U_{i}$ therefore induces an isomorphism $\pi_{1}^{\orb}(F_{i})\cong\pi_{1}^{\orb}(U_{i})$. Then by applying the orbifold Seifert--van Kampen theorem \cite[Cor.~2.3]{BoileauMaillotPorti2003} once for each end, it follows that the inclusion $Q_{0}\subseteq Q^{\times}$ also induces an isomorphism on orbifold fundamental groups.
    
    \begin{lemma}\label{ass}
        If $k=2$ and $s\ge 2$, then $Q_{0}$ is hyperbolic, that is, the orbifold $Q^{\times}\cong\Int(Q_{0})$ carries a complete hyperbolic structure.
    \end{lemma}   
    \begin{proof}
    By \cite[Thm.~9]{jonesreid} the orbifold $Q_{0}$ is hyperbolic (which is case (5) there) unless it falls into one of the four cases (1)--(4) listed in that theorem. We check that our hypotheses exclude each of them.
    \begin{enumerate}[label=\textup{(\arabic*)}]
    \item The bad case, in which $L$ is the two-component unlink. By \cite[p.~244]{jonesreid} the link $L$ is the rational link of slope $r/s$, whereas the two-component unlink is the rational link of slope $1/0$. The convention $0\le r<s$ excludes the latter.
    \item The spherical case, in which $T=\pi_{1}^{\orb}(Q_{0})$ is finite. An inspection of the classification of finite generalised triangle groups listed in \cite[Thm.~7.3.4.1]{algdiscgrp} shows that each of them has $c_{1}>0$ or $c_{2}>0$, and hence $k<2$.
    \item The Seifert-fibred case, in which $s$ is odd, $n=2$, $\min(l,m)=2$ and $2r\equiv\pm1\pmod s$. For odd $s$ the conditions $n=2$ and $\min(l,m)=2$ give $c_{1}=c_{2}=1/\max(l,m)>0$, and hence $k=0$.
    \item The $I$-fibred case, in which $r=0$ and $s=1$. This is excluded by $s\ge 2$.
    \end{enumerate}
    For a compact orbifold with boundary, hyperbolicity means that the interior carries a complete hyperbolic structure \cite[\S1.1]{BoileauMaillotPorti2003}, and the interior of $Q_{0}$ is $Q^{\times}$, as observed above.
    \end{proof}

    Jones and Reid deduce \cite[Thm.~9]{jonesreid} from Thurston's orbifold theorem \cite[Thms.~3.27, 9.1]{BoileauMaillotPorti2003}. We fix a complete hyperbolic structure on $Q^{\times}$ as in Lemma \ref{ass}, in which the end $U_{i}$ is a cusp when $c_{i}=0$ and a funnel of infinite volume when $c_{i}<0$ \cite[\S6.2.2]{BoileauMaillotPorti2003}, \cite[\S3]{GarciaPorti}.
    
    Since $Q^{\times}$ is orientable, the holonomy of this structure takes values in $\Isom^{+}(\HH^{3})=\PSL(2,\C)$. This gives a faithful representation $\rho_{0}\colon T\cong\pi_{1}^{\orb}(Q^{\times})\to\PSL(2,\C)$ with discrete image that acts properly discontinuously on $\HH^{3}$, and $Q^{\times}$ is isometric to the quotient \cite[Thm.~2.26]{CHK}. We set
    \[
    \Gamma\coloneqq \rho_{0}(T)\le \PSL(2,\C), \quad \pi\colon\HH^{3}\to \HH^{3}/\Gamma=Q^{\times}, \quad C\coloneqq \pi^{-1}(Q_{0}).
    \]
    Then $\Gamma$ acts faithfully on $\HH^{3}$, so the orbifold $Q^{\times}=\HH^{3}/\Gamma$ is good and $\pi$ is its universal orbifold covering, with deck group $\Gamma\cong T$ (\S\ref{subsec:orbifolds}). In particular the local groups of $Q^{\times}$ are the point stabilisers of the action, up to conjugacy.
    
    The next theorem describes how $Q_{0}$ sits inside $\HH^{3}/\Gamma$, together with the structure of its preimage $C$.
    \begin{theorem}\label{thm:geometry}
        Let $k=2$ and $s\geq 2$, and let $\Gamma$, $\pi$ and $C$ be as above. Then
        \begin{enumerate}[label=\textup{(\alph*)}]
            \item\label{geo:prod} $(\HH^{3}/\Gamma)\setminus Q_0$ is the disjoint union of the ends $F\times(0,\infty)$, one for each component $F$ of $\partial Q_{0}$, with the local group of $\HH^{3}/\Gamma$ at $(a,t)$ isomorphic to that of $F$ at $a$.
            \item\label{geo:retract} $C$ is $\Gamma$-invariant and contractible, and $\pi|_{C}\colon C\to C/\Gamma=Q_{0}$ is the universal orbifold covering.
        \end{enumerate}
    \end{theorem}

    \begin{proof}
        Part \ref{geo:prod} follows from the description of $Q^{\times}$ and its ends given above. Indeed, $Q^{\times}\setminus Q_{0}$ is obtained by removing the level $t=0$ from each end $U_{i}\cong F_{i}\times[0,\infty)$, and the local groups of $F_{i}\times(0,\infty)$ are those of $F_{i}$.
        
        For part \ref{geo:retract} we construct an explicit deformation retraction of $\HH^{3}$ onto $C$. Put $\Sigma\coloneqq\pi^{-1}(\partial Q_{0})$ and $P\coloneqq\pi^{-1}\bigl(\bigsqcup_{i}U_{i}\bigr)$. These are closed subsets of $\HH^{3}$ and since $Q^{\times}$ is the union of $Q_{0}$ and the ends $U_{i}$, which meet $Q_{0}$ along $\partial Q_{0}$, we have $\HH^{3}=C\cup P$ and $C\cap P=\Sigma$. Each end $U_{i}\cong F_{i}\times[0,\infty)$ retracts onto $F_{i}\times\{0\}$ by shrinking the interval, and we lift this to $P$ while keeping $C$ fixed. In order to do this we will need a diffeomorphism $P\cong\Sigma\times[0,\infty)$ which maps $\Sigma$ to $\Sigma\times\{0\}$, and we obtain such a diffeomorphism by lifting the product structure of the ends.
        
        Fix an end $U_{i}$ and a point $x_{0}\in F_{i}$ with trivial local group, and write $P_{i}\coloneqq\pi^{-1}(U_{i})$ and $\Sigma_{i}\coloneqq\pi^{-1}(F_{i})$. Since the restriction of an orbifold covering over a suborbifold is again an orbifold covering, both $\pi|_{P_{i}}$ and $\pi|_{\Sigma_{i}}\times\mathrm{id}$ are orbifold coverings of $U_{i}=F_{i}\times[0,\infty)$, with the same fibre over $x_{0}$. Every element of $\pi_{1}^{\orb}(U_{i})$ is represented by a loop in $F_{i}\times\{0\}$, and for both coverings the lifts of such a loop are its lifts to $\Sigma_{i}=\Sigma_{i}\times\{0\}$ under $\pi|_{\Sigma_{i}}$. The two monodromy actions on the fibre therefore agree, and since a possibly disconnected orbifold covering is determined up to equivalence by its monodromy action \cite[Thm.~2.2(ii)]{BoileauMaillotPorti2003}, we obtain a diffeomorphism $\Phi_{i}\colon P_{i}\to\Sigma_{i}\times[0,\infty)$ over $U_{i}$. Moreover, $\Phi_{i}$ maps $\Sigma_{i}$ onto $\Sigma_{i}\times\{0\}$, since these are the parts lying over $F_{i}\times\{0\}$. Taking the disjoint union over the ends, we obtain a diffeomorphism $\Phi\colon P\to\Sigma\times[0,\infty)$ which satisfies $\Phi(\Sigma)=\Sigma\times\{0\}$.
        
        We can now write down the retraction. For $u\in[0,1]$ let $\theta_{u}\colon P\to P$ correspond under $\Phi$ to the map $(a,t)\mapsto(a,(1-u)t)$, which shrinks the interval factor. Then $\theta_{0}=\mathrm{id}_{P}$ and the image of $\theta_{1}$ is $\Sigma$. Now using the decomposition $\HH^{3}=C\cup P$, we can extend this to all of $\HH^{3}$ by defining $
        \Theta_{u}\colon\HH^{3}\to\HH^{3}$ to be the identity on $C$ and $\theta_{u}$ on $P$. Notice that these two definitions agree on the intersection $C\cap P=\Sigma$ because every $\theta_{u}$ fixes $\Sigma$ pointwise. Then since $C\times[0,1]$ and $P\times[0,1]$ are closed sets covering $\HH^{3}\times[0,1]$, the pasting lemma implies that the map $\Theta\colon\HH^{3}\times[0,1]\to \HH^{3}$ given by $(x,u)\mapsto \Theta_{u}(x)$ is continuous. By construction $\Theta$ is the required deformation retraction of $\HH^{3}$ onto $C$, and in particular $C$ is contractible.
        
        Finally, $C$ is $\Gamma$-invariant, and $\pi|_{C}\colon C\to Q_{0}$ is an orbifold covering whose fibres are $\Gamma$-orbits. Thus $C/\Gamma=Q_{0}$, and since $C$ is simply connected, $\pi|_{C}$ is the universal orbifold covering by \S\ref{subsec:orbifolds}.
    \end{proof}
    
    \subsection{Peripheral subgroups}\label{subsec:peripheral}
    Each boundary turnover of $Q_{0}$ contributes a subgroup of $T$, namely the stabiliser in $\Gamma$ of a component of its preimage in $\HH^{3}$.
    
        \begin{proposition}\label{prop:peripheral}
        Suppose $c_{i}\le 0$, and let $F_{i}\subseteq\partial Q_{0}$ be the $i$-th boundary turnover. The components of $\pi^{-1}(F_{i})$ are pairwise disjoint, properly embedded planes in $\HH^{3}$, and $\Gamma$ permutes them transitively. The stabiliser $S_{i}=\Stab_{\Gamma}(Y)$ of a component $Y\subseteq\pi^{-1}(F_{i})$ acts properly discontinuously and cocompactly on $Y$, and $\pi$ induces an isomorphism of orbifolds $Y/S_{i}\cong F_{i}$.
    \end{proposition}
 
    \begin{proof}
        Firstly, recall that a turnover of non-positive Euler characteristic is incompressible because it carries no essential curves \cite[\S3.1]{BoileauMaillotPorti2003}. Then the equivariant loop theorem \cite[Cor.~3.20]{BoileauMaillotPorti2003} implies that the inclusion $F_{i}\subseteq\HH^{3}/\Gamma$ is $\pi_{1}$-injective, that is, it induces an injection $\pi_{1}^{\orb}(F_{i})\hookrightarrow\Gamma$. Moreover, $F_{i}$ is a closed $2$-suborbifold of $\HH^{3}/\Gamma$, so its preimage $\pi^{-1}(F_{i}) \subseteq \HH^{3}$ is a surface which is closed as a subset, and its components are pairwise disjoint and properly embedded. Further, the group $\Gamma$ acts transitively on the fibres of $\pi$, so it permutes these components transitively.
        
        Now fix a component $Y\subseteq \pi^{-1}(F_{i})$. Then by \S\ref{subsec:orbifolds} the restriction $\pi|_{Y}\colon Y\to F_{i}$ is an orbifold covering, and there is a commutative square
        \[
            \begin{tikzcd}
                Y \arrow[r, "\pi|_{Y}"] \arrow[d, hook] & F_{i} \arrow[d, hook] \\
                \HH^{3} \arrow[r, "\pi"] & \HH^{3}/\Gamma,
            \end{tikzcd}    
        \]
        where the vertical maps are inclusions. It follows that the composition of group homomorphisms $\pi_{1}(Y)\hookrightarrow\pi_{1}^{\orb}(F_{i})\hookrightarrow\Gamma$ factors through $\pi_{1}(\HH^{3})=1$, which implies that $\pi_{1}(Y)=1$, and hence $\pi|_{Y}$ is the universal orbifold covering of $F_{i}$. In particular, $Y$ is $\EE^{2}$ or $\HH^{2}$, depending on whether $F_{i}$ is Euclidean or hyperbolic, and in either case $Y$ is homeomorphic to $\R^{2}$.
         
        Let $S_{i}$ be the stabiliser of $Y$. If $y,y'\in Y$ satisfy $\pi(y)=\pi(y')$, then $y'=\gamma y$ for some $\gamma\in\Gamma$, and $\gamma\in S_{i}$ because $\gamma Y$ is a component meeting $Y$. It follows that $\pi|_{Y}$ induces a homeomorphism $Y/S_{i}\to F_{i}$, and $S_{i}$ acts properly discontinuously and cocompactly on $Y$ because $\Gamma$ acts properly discontinuously and $F_{i}$ is compact. For $y\in Y$ the stabiliser $\Gamma_{y}$ fixes $y$ and therefore preserves $Y$. Then $\Gamma_{y}$ is the stabiliser of $y$ in $S_{i}$, and it is also the local group at $\pi(y)$ by \S\ref{subsec:orbifolds}. In this way the homeomorphism $Y/S_{i}\to F_{i}$ identifies local groups, and hence it is an isomorphism of orbifolds.
    \end{proof}
 
    \begin{corollary}\label{cor:peripheral}
        Suppose $c_{i}\le 0$. Then the stabilisers of the components of $\pi^{-1}(F_{i})$ form a single conjugacy class of subgroups of $T\cong\Gamma$, each isomorphic to $\pi_{1}^{\orb}(F_{i})$. The group $\pi_{1}^{\orb}(F_{i})$ is a 2-dimensional crystallographic group when $c_{i}=0$ and a cocompact Fuchsian triangle group when $c_{i}<0$.
    \end{corollary}
 
    \begin{proof}
        Let $S_{i}$ be the stabiliser of a component $Y$. Then since $Y\cong\R^{2}$ is simply connected, the covering $Y\to Y/S_{i}\cong F_{i}$ of Proposition \ref{prop:peripheral} is universal with deck group $S_{i}$, and hence $S_{i}\cong\pi_{1}^{\orb}(F_{i})$ (\S\ref{subsec:orbifolds}). Now because the stabiliser of $\g Y$ is $\g S_{i}\g^{-1}$ and $\Gamma$ permutes the components transitively, it follows that the stabilisers form a single conjugacy class. For $c_{i}=0$ the group $\pi_{1}^{\orb}(F_{i})$ is $p3$, $p4$ or $p6$ by Lemma \ref{lem:euclidean-links}, and for $c_{i}<0$ it is a cocompact Fuchsian triangle group \cite[\S7.2]{Ratcliffe2006}.
    \end{proof}

    These peripheral subgroups are used in Section \ref{sec:vcdT} to bound the virtual cohomological dimension of $T$ from below and in Section \ref{sec:H2} to count the boundary components of a manifold cover of $Q_{0}$.

    \section{Virtual cohomological dimension of \texorpdfstring{$T$}{T}}\label{sec:vcdT}
    We compute $\vcd(T)$ in two steps: Selberg's lemma \cite[Thm.~7.6.7]{Ratcliffe2006} lets us pass to a torsion-free subgroup of finite index, which gives an upper bound, and the peripheral subgroups of Proposition \ref{prop:peripheral} give a lower bound.

    \begin{lemma}\label{lem:core}
        Let $T_{0}$ be a torsion-free normal subgroup of finite index in $\Gamma=\rho_{0}(T)\cong T$. Then $M_{0}\coloneqq C/T_{0}$ is a compact orientable aspherical $3$-manifold with $\pi_{1}(M_{0})\cong T_{0}$. Further, the boundary $\partial M_{0}$ is the union, over the $i$ with $c_{i}\le 0$, of the closed orientable surfaces $\pi^{-1}(F_{i})/T_{0}$.
    \end{lemma}
 
    \begin{proof}
        The group $T_{0}$ is a torsion-free discrete subgroup of $\PSL(2,\C)$, so it acts freely and properly discontinuously on $\HH^{3}$, and hence on the $\Gamma$-invariant subset $C=\pi^{-1}(Q_{0})$. This subset is contractible by Theorem \ref{thm:geometry}\ref{geo:retract}, so $M_{0}=C/T_{0}$ is an aspherical orientable manifold with universal cover $C$ and $\pi_{1}(M_{0})\cong T_{0}$. Furthermore, because $Q_{0}$ is compact and the index $[\Gamma:T_{0}]$ is finite, it follows that $M_{0}$ is compact.
        
        Now suppose $c_{i}\le 0$ and let $F_{i}\subseteq\partial Q_{0}$ be the $i$-th boundary turnover. By Proposition \ref{prop:peripheral} the components of $\pi^{-1}(F_{i})$ form one $\Gamma$-orbit with stabiliser $S_{i}$. It follows that the components of $\partial M_{0}$ lying over $F_{i}$ are indexed by the double cosets $T_{0}\backslash \Gamma/S_{i}$, and $\Pi\coloneqq \Gamma/T_{0}$ acts transitively on them. Each of them has the form $\g Y/(T_{0}\cap\Stab_\Gamma(\g Y))$ for some $\g\in\Gamma$ and some component $Y$ of $\pi^{-1}(F_{i})$. Since $T_{0}$ is normal, this quotient is homeomorphic to $Y/(S_{i}\cap T_{0})$, and thus it will suffice to treat this case.

        The subgroup $S_{i}\cap T_{0}$ is torsion-free and has finite index in $S_{i}$, and therefore acts on $Y$ freely, properly discontinuously and cocompactly. Then the quotient $Y/(S_{i}\cap T_{0})$ is a closed aspherical surface which is orientable because it covers the orientable turnover $F_{i}$, and has Euler characteristic $[S_{i}:S_{i}\cap T_{0}]\,c_{i}$. In particular, the quotient is a torus when $c_{i}=0$ and has genus at least two when $c_{i}<0$.
    \end{proof}
     
    \begin{theorem}\label{thm:vcdT}
        Suppose $k=2$ and $s\ge 2$. Then $\vcd(T)=2$.
    \end{theorem}
     
    \begin{proof}
        Let $T_{0}$ and $M_{0}$ be as in Lemma \ref{lem:core}, so that $\vcd(T)=\cd(T_0)$ and $M_{0}$ is a compact orientable aspherical $3$-manifold with $\pi_{1}(M_{0})=T_{0}$. Since $k=2$, the boundary $\partial M_{0}$ is non-empty by Lemma \ref{lem:core}. A compact $3$-manifold with non-empty boundary collapses onto a $2$-dimensional subcomplex \cite[proof of Thm.~3]{ratcliffe}, and since $M_{0}$ is aspherical this subcomplex is a $K(T_{0},1)$. Hence $\cd(T_{0})\le 2$. 
        
        For the lower bound fix $i\in\{1,2\}$ with $c_{i}\le 0$ and let $Y/(S_{i}\cap T_{0})$ be as in the proof of Lemma \ref{lem:core}. This quotient is a closed aspherical surface with fundamental group $S_{i}\cap T_{0}$, and hence $\cd(S_{i}\cap T_{0})=2$. Then since $S_{i}\cap T_{0}$ is a subgroup of $T_{0}$, it follows that $\cd(T_{0})\ge 2$ \cite[Ch.~VIII, Prop.~2.4(a)]{brown}, and hence $\vcd(T)=\cd(T_{0})=2$.
    \end{proof}

    \begin{corollary}\label{cor:Tvfl}
        Suppose $k=2$ and $s\ge 2$. The group $T$ is of type $\VFL$.
    \end{corollary}
    \begin{proof}
        By Lemma \ref{lem:core} the group $T$ has a finite-index subgroup $T_{0}$ which is the fundamental group of a compact aspherical $3$-manifold. The result then follows from \cite[Ch.~VIII, Prop.~8.1]{brown}.
    \end{proof}

    \section{Abelianisations}\label{sec:ab}
    
    Both $\Tab$ and $\Gab$ are quotients of $\Z^{2}=\langle e_{1},e_{2}\rangle$, where $e_{1},e_{2}$ are the images of $f,g$, and for $w\in\Z^{2}$ we write $\bar w$ for its class in the quotient at hand. In particular, $\bar c$ will denote the class of $c=f^{l}$, that is of $le_{1}$, in $\Gab$. Since $T=G/\langle c\rangle$, we have
    \[
        \Tab=\Gab\big/\langle\bar c\rangle, \quad\text{so}\quad |\Gab|=|\Tab|\cdot\ord(\bar c)
    \]
    whenever these are finite.
    
    For an integer matrix $R$ whose rows span a subgroup $\Lambda\le\Z^{2}$, let $\mu_{1}(R)$ be the gcd of the entries of $R$ and $\mu_{2}(R)$ the gcd of its $2\times 2$ minors, with $\mu_{2}(R)=0$ if these all vanish. Then by the theory of the Smith normal form \cite[\S2.4.4]{smithnorm}, we have
    \[
        \Z^{2}/\Lambda\cong\Z/\mu_{1}(R)\oplus\Z/\bigl(\mu_{2}(R)/\mu_{1}(R)\bigr),
    \]
    with the convention that $\Z/0\coloneqq\Z$. We also use the fact that, for $w\in\Z^{2}$ with $\mu_{2}(R)\ne 0$,
    \[
        \ord(\bar w)=[\Lambda+\Z w:\Lambda] =\frac{[\Z^{2}:\Lambda]}{[\Z^{2}:\Lambda+\Z w]} =\frac{\mu_{2}(R)}{\mu_{2}(R_w)},
    \]
    where $R_{w}$ denotes $R$ with the row $w$ appended.
    
    \begin{proposition}\label{prop:Tab}
        We have $\Tab=\Z^{2}/\Lambda_{T}$, where $\Lambda_{T}$ is spanned by the rows of
        \[
          R_{T}=\begin{bmatrix} 
                l & 0\\ 
                0 & m\\ 
                n\sigma_{f} & n\sigma_{g}
          \end{bmatrix}.
        \]
        Here $\mu_{1}(R_{T})=\nu$ and $\mu_{2}(R_{T})\ne 0$, so $\Tab$ is finite with invariant factors $\nu$ and $\mu_{2}(R_{T})/\nu$. Explicitly, we have
        \[
            \Tab\cong
            \begin{cases}
                \Z/\gcd(l,m)\oplus\Z/\lcm(l,m)\cong\Z/l\oplus\Z/m, & s \textup{ even},\ |\Tab|=lm,\\[2pt]
                \Z/\gcd(l,m,n)\oplus\Z/\bigl(\gcd(lm,ln,mn)/\gcd(l,m,n)\bigr), & s \textup{ odd},\ |\Tab|=\gcd(lm,ln,mn).
            \end{cases}
        \]
        In particular $\Hom(\Tab,\Q)=0$, and $\Tab$ is cyclic if and only if $\nu=1$.
    \end{proposition}
    
    \begin{proof}
        Abelianising the presentation of $T$, we can identify $\Tab$ with $\Z^{2}$ modulo the exponent-sum vectors of the relators $f^{l}$, $g^{m}$ and $\W^{n}$, and these vectors are the rows of $R_{T}$. The gcd of these entries is $\mu_{1}(R_{T})=\gcd(l,m,n\sigma_{f},n\sigma_{g})=\nu$, and the three $2\times2$ minors are $lm$, $ln\sigma_{g}$ and $-mn\sigma_{f}$. For $s$ even the exponent sums $\sigma_{f}$ and $\sigma_{g}$ vanish and only the first minor survives, giving $\mu_{2}(R_{T})=lm$. For $s$ odd these exponent sums are units, and $\mu_{2}(R_{T})=\gcd(lm,ln,mn)$. In either case $\mu_{2}(R_{T})\neq 0$, and the Smith normal form yields the two displayed groups, where in the even case we have used $lm/\gcd(l,m)=\lcm(l,m)$.
    \end{proof}
    
    \begin{proposition}\label{prop:Gab}
        We have $\Gab=\Z^{2}/\Lambda_{G}$, where $\Lambda_{G}$ is spanned by the rows of
        \begin{equation}\label{eq:J}
            J=\begin{bmatrix} 
              n\sigma_{f}-l & n\sigma_{g}\\ 
              n\sigma_{f} & n\sigma_{g}-m
            \end{bmatrix},
            \quad
            \det J=lm-n(l\sigma_{g}+m\sigma_{f}),
        \end{equation}
        and $\mu_{1}(J)=\nu$ in all cases. Consequently the following hold.
        \begin{enumerate}[label=\textup{(\alph*)}]
            \item If $s$ is even, then $\det J=lm$ and $\Gab\cong\Z/\gcd(l,m)\oplus \Z/\lcm(l,m)$. Moreover $\bar c=0$, so that $\Gab\to\Tab$ is an isomorphism.
            
            \item If $s$ is odd and $\det J\ne 0$, then
              \[
                \Gab\cong\Z/\nu\oplus\Z/\tfrac{|\det J|}{\nu},
                \quad
                \nu=\gcd(l,m,n),
                \quad
                \ord(\bar c)=\frac{|\det J|}{\gcd\bigl(\det J,\,l\gcd(m,n)\bigr)}.
              \]
              
            \item If $s$ is odd and $\det J=0$, then $\Gab\cong\Z/\nu\oplus\Z$. Moreover, the class $\bar c$ has infinite order, and its image in the free quotient generates a subgroup of index $\lcm(l,m)$.
        \end{enumerate}
    \end{proposition}
    
    \begin{proof}
        For the relators of $G$ we use $W^{n}f^{-l}$ and $W^{n}g^{-m}$, whose exponent-sum vectors are the rows of $J$. It follows that $\mu_{1}(J)=\gcd(n\sigma_{f}-l,\; n\sigma_{g},\; n\sigma_{f},\; n\sigma_{g}-m)$. For $s$ even this is $\gcd(l,0,0,m)=\gcd(l,m)$. For $s$ odd the entries $n\sigma_{f}$ and $n\sigma_{g}$ have absolute value $n$, so the subgroup of $\Z$ generated by the four entries contains $n$ and hence also $(n\sigma_{f}-l)-n\sigma_{f}=-l$ and $(n\sigma_{g}-m)-n\sigma_{g}=-m$. Conversely $n$, $l$ and $m$ generate all four entries. Hence $\mu_{1}(J)=\gcd(l,m,n)$, and in both cases $\mu_{1}(J)=\nu=\mu_{1}(R_T)$.

        \begin{enumerate}[label=\textup{(\alph*)}]
            \item If $s$ is even, then $\sigma_{f}=\sigma_{g}=0$, so $\det J=lm\ne 0$ and the Smith normal form gives the group. Since the first row of $J$ is $(-l,0)$, we have $-le_{1}\in\Lambda_{G}$ and hence $\bar c=0$. Thus $\Gab\to\Gab/\langle \bar c\rangle=\Tab$ is an isomorphism.
    
            \item The Smith normal form gives the group, since $\mu_{2}(J)=|\det J|$ and $\mu_{1}(J)=\nu$. To find $\ord(\bar c)$ we apply the formula for $\ord(\bar w)$ above with $w=le_{1}$. The extra $2\times 2$ minors of $J_{w}$ are $-ln\sigma_{g}$ and $-l(n\sigma_{g}-m)$, and $|\sigma_{g}|=1$ in this case. Thus $\mu_{2}(J_{w})=\gcd(|\det J|,lm,ln)=\gcd(\det J,l\gcd(m,n))$.
    
            \item Now $\mu_{2}(J)=0$ and $\mu_{1}(J)=\nu\ne 0$, and the Smith normal form gives $\Gab\cong\Z/\nu\oplus\Z$. For the class of $c$ consider the map (cf.\ \cite[proof of Lem.~1]{ellistri})
            \begin{equation*}\label{eq:psi}
              \psi\colon\Z^{2}\to\Z,\quad \psi(x,y)=\frac{mx+ly}{\eta} .
            \end{equation*}
            This map is surjective, since $\psi(1,0)=m/\eta$ and $\psi(0,1)=l/\eta$ are coprime. The kernel of $\psi$ is $\Z w$ with $w=(l,-m)/\eta\in\Z^{2}$ primitive, and $\psi$ vanishes on both rows of $J$, since
            \[
              \psi(n\sigma_{f}-l,n\sigma_{g}) = \psi(n\sigma_{f},n\sigma_{g}-m) = \frac{n(m\sigma_{f}+l\sigma_{g})-lm}{\eta} = -\frac{\det J}{\eta} = 0.
            \] 
            Then $\psi$ descends to a surjection $\bar\psi\colon\Gab\to\Z$ with finite kernel $\Z w/\Lambda_{G}$. It follows that this kernel is the torsion subgroup $\Z/\nu$ of $\Gab$, and $\bar\psi$ is the projection onto the free quotient. Finally, $\bar\psi(\bar c)=\psi(l,0)=ml/\eta=\lcm(l,m)\ne 0$, so $\bar c$ has infinite order and its image generates a subgroup of index $\lcm(l,m)$ in $\bar\psi(\Gab)=\Z$. \qedhere
        \end{enumerate}
    \end{proof}
    
    The map $\psi$ from the proof of Proposition \ref{prop:Gab}(c) also gives a direct proof that $c$ has infinite order in $G$ when $\det J=0$.
    \begin{lemma}\label{lem:Psi}
        Suppose $\det J=0$. Then
        \[
            \Psi\colon G\to\Z, \quad \Psi(f)=\frac{m}{\eta}, \quad \Psi(g)=\frac{l}{\eta}
        \]
        is a well-defined homomorphism with $\Psi(c)=\lcm(l,m)\ne 0$. In particular $\langle c\rangle\cong\Z$.
    \end{lemma}
    
    \begin{proof}
        Since $\Psi(W)=\sigma_{f}\Psi(f)+\sigma_{g}\Psi(g)=(m\sigma_{f}+l\sigma_{g})/\eta$, it follows that both relators of $G$ are respected. Indeed, we have
        \[
            \Psi(W^{n}f^{-l}) = \Psi(W^{n}g^{-m}) = \frac{n(m\sigma_{f}+l\sigma_{g})-lm}{\eta} = -\frac{\det J}{\eta}=0,
        \]
        and $\Psi(c)=\Psi(f^{l})=lm/\eta=\lcm(l,m)\neq0$.
    \end{proof}

    Finally we record that the degenerate case $\det J=0$ occurs only under the hypothesis that $k=2$.

    \begin{lemma}\label{lem:detJk2}
        If $\det J=0$, then $s$ is odd and $k=2$.
    \end{lemma}

    \begin{proof}
        For $s$ even, $\det J=lm\ne 0$ by Proposition \ref{prop:Gab}(a), so $s$ is odd and $\sigma_{f},\sigma_{g}\in\{\pm1\}$. Now $\det J=0$ is equivalent to $lm=n(l\sigma_{g}+m\sigma_{f})$, and the right-hand side is negative when $\sigma_{f}=\sigma_{g}=-1$. In the three remaining cases $(\sigma_{f},\sigma_{g})=(1,1),(1,-1),(-1,1)$, dividing by $lmn$ we obtain
        \[
          \frac{1}{n}=\frac{1}{l}+\frac{1}{m},\qquad \frac{1}{n}=\frac{1}{l}-\frac{1}{m},\qquad \frac{1}{n}=\frac{1}{m}-\frac{1}{l},
        \]
        respectively. It follows that $\frac{1}{l}+\frac{1}{m}+\frac{1}{n}$ equals $\frac{2}{n}$, $\frac{2}{l}$ or $\frac{2}{m}$ respectively, and each of these is at most $1$ because we are assuming that $l,m,n\ge2$, so that $c_{1}=c_{2}\le0$ and $k=2$.
    \end{proof}
    
    \section{The second homology group \texorpdfstring{$\Ho_{2}(T,\Q)$}{H2(T)}}\label{sec:H2}
    We compute the rational second homology of $T$ from the manifold $M_{0}$ of Lemma \ref{lem:core}. Section \ref{sec:central} then uses the result to show that the central element $c$ has infinite order in $G$.

    \begin{theorem}\label{thm:H2Q}
        Suppose $k=2$ and $s\ge 2$. Then $\Ho_{2}(T,\Q)\cong\Q$ and $\Ho_{2}(T,\Z)$ is finitely generated of rank one.
    \end{theorem}
    
    \begin{proof}
        As in Lemma \ref{lem:core}, let $T_{0}\trianglelefteq T$ be a torsion-free normal subgroup of finite index, and let $M_{0}$ be the associated manifold.
        
        For the finite group $\Pi=T/T_{0}$ and any $\Q\Pi$-module $V$, rational transfer gives $\Ho_{i}(\Pi,V)=0=\Ho^{i}(\Pi,V)$ for $i>0$ \cite[Ch.~III, Cor.~10.2]{brown}. The Lyndon--Hochschild--Serre spectral sequences of the extension $1\to T_{0}\to T\to\Pi\to1$ therefore collapse. Moreover, since $M_{0}$ is aspherical with fundamental group $T_{0}$, we obtain
        \[
          \Ho_{*}(T,\Q)\cong \Ho_{*}(M_{0},\Q)^{\Pi},
          \quad
          \Ho^{*}(T,\Q)\cong \Ho^{*}(M_{0},\Q)^{\Pi} ,
        \]
        where invariants and coinvariants agree because $|\Pi|$ is invertible in $\Q$. By Proposition \ref{prop:Tab} this gives
        \[
          \Ho^{1}(M_{0},\Q)^{\Pi}=\Ho^{1}(T,\Q)=\Hom(\Tab,\Q)=0.
        \]
        The group $\Pi$ acts on $M_{0}$ by orientation-preserving diffeomorphisms, since these are induced by elements of $\Gamma\le\PSL(2,\C)$. In particular $\Pi$ fixes the fundamental class, and the duality isomorphisms below are $\Q\Pi$-isomorphisms.
        
        Let $\mathcal B$ be the set of components of $\partial M_{0}$, which is non-empty and finite by Lemma \ref{lem:core}. Lefschetz duality, given by cap product with the fundamental class in $\Ho_{3}(M_{0},\partial M_{0},\Q)$, identifies $\Ho_{2}(M_{0},\Q)$ with $\Ho^{1}(M_{0},\partial M_{0},\Q)$ \cite[Thm.~3.43]{Hatcher}. We compute the latter from the long exact cohomology sequence of the pair $(M_{0},\partial M_{0})$ \cite[\S3.1]{Hatcher}, whose low-degree part reads
        \[
            \begin{tikzcd}[sep=1.2em, cramped]
    			0 \arrow[r] & \Ho^{0}(M_{0},\partial M_{0}) \arrow[r] & \Ho^{0}(M_{0}) \arrow[r, "j^{0}"] & \Ho^{0}(\partial M_{0}) \arrow[r, "\delta"] & \Ho^{1}(M_{0},\partial M_{0}) \arrow[r] & \Ho^{1}(M_{0}) \arrow[r, "j^{1}"] & \Ho^{1}(\partial M_{0}) \arrow[r] & \cdots
    	   \end{tikzcd}
        \]
        with $\Q$ coefficients throughout, where $j^{0}$ and $j^{1}$ are the restrictions induced by $\partial M_{0}\hookrightarrow M_{0}$. Since $M_{0}$ is connected we have $\Ho^{0}(M_{0},\Q)=\Q$, while $\Ho^{0}(\partial M_{0},\Q)=\Q[\mathcal B]$ is the permutation module on $\mathcal B$ and $j^{0}$ is the diagonal. Now a locally constant function on $M_{0}$ which vanishes on the non-empty boundary $\partial M_{0}$ must vanish identically, so $j^{0}$ is injective and $\Ho^{0}(M_{0},\partial M_{0},\Q)=0$. Then exactness identifies $\im(\delta)$ with $\coker(j^{0})\cong\Q[\mathcal B]/\Q$, and the image of $\Ho^{1}(M_{0},\partial M_{0},\Q)\to\Ho^{1}(M_{0},\Q)$ with $\ker(j^{1})$. It follows that the sequence reduces to
        \begin{equation*}\label{eq:pair-ses}
            \begin{tikzcd}
    			0 \arrow[r] & \Q[\mathcal B]/\Q \arrow[r] & \Ho^{1}(M_{0},\partial M_{0}, \Q) \arrow[r] & \ker(j^{1}) \arrow[r] & 0.
    	   \end{tikzcd}
        \end{equation*}
        By Maschke's theorem the functor $(-)^{\Pi}$ is exact on $\Q\Pi$-modules. Applying it, and using duality and the transfer isomorphism, we obtain
        \[
          \begin{tikzcd}
    			0 \arrow[r] & (\Q[\mathcal B]/\Q)^{\Pi} \arrow[r] & \Ho_{2}(T,\Q) \arrow[r] & (\ker(j^{1}))^{\Pi} \arrow[r] & 0.
    	   \end{tikzcd}
        \]
        The right-hand term vanishes, because $(\ker(j^{1}))^{\Pi}$ is contained in $\Ho^{1}(M_{0},\Q)^{\Pi}=\Hom(\Tab,\Q)=0$. For the left-hand term, exactness of $(-)^{\Pi}$ gives $(\Q[\mathcal B]/\Q)^{\Pi}=\Q[\mathcal B]^{\Pi}/\Q$, and $\dim_{\Q}\Q[\mathcal B]^{\Pi}$ is the number of $\Pi$-orbits in $\mathcal B$. As in the proof of Lemma \ref{lem:core}, $\Pi$ acts transitively on the components of the boundary $\partial M_{0}$ lying over $F_{i}$. Then since $\partial Q_{0}$ has $k=2$ components by \S\ref{subsec:structures}, there are two $\Pi$-orbits in $\mathcal B$ and we have
        \[
            \Ho_{2}(T,\Q)\cong\Ho_{2}(M_{0},\Q)^{\Pi}\cong\Ho^{1}(M_{0},\partial M_{0},\Q)^{\Pi}\cong(\Q[\mathcal B]/\Q)^{\Pi}=\Q[\mathcal B]^{\Pi}/\Q\cong\Q.
        \]
        Finally, $T$ is finitely presented. By \cite[Ch.~II, Thm.~5.2]{brown} the group $\Ho_{2}(T,\Z)$ is a quotient of the second homology of a finite presentation $2$-complex, and hence finitely generated. Since $\Q$ is flat over $\Z$ we have $\Ho_{2}(T,\Q)\cong \Ho_{2}(T,\Z)\otimes\Q$, which identifies the rank.
    \end{proof}
 
    \section{The central subgroup and \texorpdfstring{$\vcd(G)$}{vcd(G)}} \label{sec:central}
    
    We now show that $c$ has infinite order in $G$. The main tool is the five-term exact sequence of the central extension $1\to\langle c\rangle\to G\to T\to 1$ \cite[Ch.~VII, Cor.~6.4]{brown}, which reads
    \begin{equation}\label{eq:5term}
        \begin{tikzcd}
    			\Ho_2(G,\Z) \arrow[r, "\alpha"] & \Ho_{2}(T,\Z) \arrow[r, "\xi"] & \langle c\rangle \arrow[r, "\iota"] & \Gab \arrow[r] & \Tab \arrow[r] & 0.
    	\end{tikzcd}
    \end{equation}
    Since $c\in G$ is central, the action of $T$ on $\Ho_{1}(\langle c\rangle,\Z)\cong\langle c\rangle$ is trivial, and hence the third term is $\langle c\rangle$ itself rather than a proper quotient. Notice also that exactness at $\Gab$ recovers the identity $\Tab=\Gab/\langle\bar c\rangle$ of Section \ref{sec:ab}, since $\im(\iota)=\langle\bar c\rangle$.
    
    Let $X_{G}$ be the presentation $2$-complex of $G$, with one $0$-cell, two $1$-cells $f,g$ and two $2$-cells attached along the relators $\W^{n}f^{-l}$ and $\W^{n}g^{-m}$. In degree two the cellular boundary map $\Z^{2}\to\Z^{2}$ is the matrix of the exponent sums of these relators, i.e., $J^{\mathsf T}$, and since $X_{G}$ has no 3-cells, we have $\Ho_{2}(X_{G},\Z)=\ker(J^{\mathsf T})$. Moreover, a $K(G,1)$ is obtained from $X_{G}$ by attaching cells of dimension at least $3$, and this cannot create $2$-cycles. It then follows from \cite[Ch.~II, Thm.~5.2]{brown} that the map $\Ho_{2}(X_{G},\Z)\to\Ho_{2}(G,\Z)$ is surjective and we have
    \[
        \det J\ne 0\implies \Ho_{2}(G,\Z)=0\implies\alpha=0 \implies\xi \text{ injective}.
    \]
    If instead $\det J=0$, then $J$ has rank one, since its first row is non-zero. In this case $\Ho_{2}(X_{G},\Z)\cong\Z$ and $\Ho_{2}(G,\Z)$ is cyclic.
    
    \begin{theorem}\label{thm:central}
        Suppose $k=2$ and $s\ge 2$. Then $\langle c\rangle\cong\Z$ and $\Ho_{2}(T,\Z)\cong\Z$. More precisely, the following hold in the five-term exact sequence \eqref{eq:5term}.
        \begin{enumerate}[label=\textup{(\alph*)}]
            \item If $\det J\ne 0$, then $\Ho_{2}(G,\Z)=0$ and $\xi$ maps $\Ho_{2}(T,\Z)$ isomorphically onto the subgroup of index $\ord(\bar c)$ in $\langle c\rangle$. For $s$ even this subgroup is all of $\langle c\rangle$.
            \item If $\det J=0$, then $s$ is odd, $\xi=0$ and $\alpha$ is an isomorphism of infinite cyclic groups.
        \end{enumerate}
    \end{theorem}
    
    \begin{proof}
        Throughout, $\Ho_{2}(T,\Z)$ is finitely generated of rank one by Theorem \ref{thm:H2Q}.
        \begin{enumerate}[label=\textup{(\alph*)}]
            \item If $\det J\ne0$, then the discussion of $X_{G}$ above gives $\Ho_{2}(G,\Z)=0$, and thus $\alpha=0$ and $\xi$ is injective. The group $\langle c\rangle$ therefore contains a copy of the infinite group $\Ho_{2}(T,\Z)$, and hence $\langle c\rangle\cong\Z$. Now by Proposition \ref{prop:Gab}(a),(b), the group $\Gab$ is finite and the subgroup $\im(\iota)=\langle\bar c\rangle$ has order $\ord(\bar c)$, so that $\ker(\iota)$ has index $\ord(\bar c)$ in $\langle c\rangle$. Exactness at $\langle c\rangle$ then identifies $\Ho_{2}(T,\Z)$ with $\ker(\iota)$, which is a subgroup of finite index in $\Z$. For $s$ even, Proposition \ref{prop:Gab}(a) gives $\bar c=0$, so that $\ord(\bar c)=1$, and it follows that $\xi$ is surjective.
            
            \item Here $s$ is odd by Proposition \ref{prop:Gab}(a), and Lemma \ref{lem:Psi} gives $\langle c\rangle\cong\Z$ directly. The class $\bar c$ has infinite order in $\Gab$ by Proposition \ref{prop:Gab}(c), so that $\iota$ is injective. Exactness at $\langle c\rangle$ then forces $\im(\xi)=\ker(\iota)=0$, and exactness at $\Ho_{2}(T,\Z)$ implies that $\alpha$ is surjective. Finally, we saw above that $\Ho_{2}(G,\Z)$ is cyclic when $\det J=0$ and, since $\Ho_{2}(T,\Z)$ has rank one, it follows that the surjection $\alpha$ is an isomorphism of infinite cyclic groups.
            \qedhere
        \end{enumerate}
    \end{proof}
    
    \begin{theorem}\label{thm:vcdG}
        Suppose $k=2$ and $s\ge 2$. Then $\vcd(G)=3$.
    \end{theorem}
    
    \begin{proof}
        Let $T_{0}\le T$ be torsion-free of finite index with $\cd(T_{0})=2$ as in Theorem \ref{thm:vcdT}, and let $G_{0}\le G$ be its preimage. Theorem \ref{thm:central} gives a central extension $1\to\Z\to G_{0}\to T_{0}\to1$ with $\Z=\langle c\rangle$. A torsion element of $G_{0}$ maps to a torsion element of $T_{0}$, which is trivial. Then such an element lies in $\langle c\rangle\cong\Z$, and hence $G_{0}$ is torsion-free. By \cite[Ch.~VIII, Prop.~2.4]{brown} we have the bound $\cd(G_{0})\le\cd(\Z)+\cd(T_{0})=3$. 
        
        For the reverse inequality we use the Lyndon--Hochschild--Serre spectral sequence of the central extension $1\to\langle c\rangle\to G_{0}\to T_{0}\to1$ with coefficients $\ZG_{0}$, which is free as a $\Z\langle c\rangle$-module. Then $\Ho^{0}(\langle c\rangle,\Z G_{0})=0$ and $\Ho^{1}(\langle c\rangle,\Z G_{0})\cong\Z T_{0}$, so that $E_{2}^{p,0}=0$, $E_{2}^{p,1}=\Ho^{p}(T_{0},\Z T_{0})$ and
        \begin{equation*}\label{eq:shift}
            \Ho^{i}(G_{0},\ZG_{0})\cong \Ho^{i-1}(T_{0},\ZT_{0}) \qquad(i\ge 1).
        \end{equation*}
        In particular $\Ho^{3}(G_{0},\ZG_{0})\cong\Ho^{2}(T_{0},\ZT_{0})$, which is non-zero because $T_{0}$ is of type $\FL$ by Corollary \ref{cor:Tvfl} with $\cd(T_{0})=2$ \cite[Ch.~VIII, Prop.~6.7]{brown}. Hence $\cd(G_{0})=3$.
    \end{proof}

    The hypothesis $k=2$ is also necessary for the conclusion $\langle c\rangle\cong\Z$, as the following remark shows.
    
    \begin{remark}\label{rem:badcases}
        Suppose that $k\le1$, and that $Q_{0}$ is hyperbolic or $T$ is finite. Then $\det J\ne0$ by Lemma \ref{lem:detJk2}, and $\Gab$ is finite by Proposition \ref{prop:Gab}. We claim that $\Ho_{2}(T,\Q)=0$ as well. Firstly, this is clear when $T$ is finite. Now if $Q_{0}$ is hyperbolic, then the constructions of \S\ref{subsec:structures} and Lemma \ref{lem:core} apply unchanged, since they only use the fact that $Q^{\times}$ is hyperbolic. For $k=0$ the manifold $M_{0}$ is closed, and the transfer argument of Theorem \ref{thm:H2Q} together with Poincar\'e duality gives $\Ho_{2}(T,\Q)\cong\Ho^{1}(T,\Q)=0$. For $k=1$ the proof of Theorem \ref{thm:H2Q} goes through with a single boundary component, and it gives $\Ho_{2}(T,\Q)=0$. Hence in all cases the finitely generated group $\Ho_{2}(T,\Z)$ is finite, and the five-term exact sequence \eqref{eq:5term} shows that $\langle c\rangle$ is finite.
        
        If $\langle c\rangle$ has finite order $d\neq 1$, then $(1+c+\dots+c^{d-1})(c-1)=0$, and hence $c-1$ is a zero-divisor in $\ZG$. Moreover $f^{ld}=c^{d}=1$, so that $G$ has torsion and $\cd(G)=\infty$ \cite[Ch.~VIII, Cor.~2.5]{brown}. As a result, neither of the resolutions of Section \ref{sec:freeres} is available in this case. However, in the case of the classical word $W=fg$, resolutions for both the triangle group and its central extensions have been constructed by Strebel \cite{streb}.
    \end{remark}
    
    \section{Cohomological dimension of \texorpdfstring{$G$}{G}} \label{sec:cdG}
    By Theorem \ref{thm:vcdG} we have $\vcd(G)=3$, and hence $\cd(G)=3$ provided that $G$ is torsion-free. To prove this we first locate the torsion of $T$.
    
    The elements $f$, $g$ and $\W$ have orders $l$, $m$ and $n$ in $T$, as Jones and Reid record \cite[p.~249]{jonesreid}. This can also be read off from Section \ref{sec:orbifold}. Indeed, the local groups $\Z/l$, $\Z/m$ and $\Z/n$ along the edges of the singular locus are generated by conjugates of $f$, $g$ and $\W$ (\S\ref{subsec:graph}), and they inject into $\Gamma\cong T$ because $\HH^{3}/\Gamma$ is good (\S\ref{subsec:orbifolds}).

    \begin{lemma}\label{lem:torsion-class}
        Suppose $k=2$ and $s\ge 2$. Then every finite subgroup of $T$ is conjugate to a subgroup of $\langle f\rangle\cong\Z/l$, $\langle g\rangle\cong\Z/m$ or $\langle\W\rangle\cong\Z/n$. In particular, every torsion element of $T$ is conjugate to a power of $f$, $g$ or $\W$.
    \end{lemma}

    \begin{proof}
        Firstly, identify $T$ with its image $\Gamma=\rho_{0}(T)\le\PSL(2,\C)$. Then since $\Gamma$ acts properly discontinuously on $\HH^{3}$, every point stabiliser is finite. On the other hand, a finite subgroup $K\le\Gamma$ fixes the circumcentre of any of its orbits \cite[Ch.~II, Cor.~2.8]{BridsonHaefliger}, and hence lies in a point stabiliser. It follows that the finite subgroups of $\Gamma$ are the subgroups of the point stabilisers, which are the local groups of $\HH^{3}/\Gamma$.

        Now by Theorem \ref{thm:geometry}\ref{geo:prod} the orbifold $\HH^{3}/\Gamma$ is $Q_{0}$ with a product end $F\times(0,\infty)$ attached along each boundary turnover $F$. The local groups along such an end are those of $F$, which are the local groups of the three branches of $L\cup\tau$ meeting $F$ (\S\ref{subsec:turnovers}). Since the singular locus of $Q_{0}$ has no vertices when $k=2$ (\S\ref{subsec:structures}), the local groups of $Q_{0}$ are trivial or cyclic edge groups. Hence every local group of $\HH^{3}/\Gamma$ is conjugate into $\langle f\rangle$, $\langle g\rangle$ or $\langle\W\rangle$, of orders $l$, $m$ and $n$. Finally, a torsion element lies in a finite cyclic subgroup, which gives the last assertion.
    \end{proof}
    
    \begin{proposition}\label{prop:G-torsionfree}
        Suppose $k=2$ and $s\ge 2$. Then $G$ is torsion-free and $\cd(G)=3$.
    \end{proposition}
    
    \begin{proof}
        By Theorem \ref{thm:central} we have $\langle c\rangle\cong\Z$. Then since $f^{l}=c$ has infinite order, so does $f$, and $\langle f\rangle\cong\Z$ contains $\langle c\rangle$ with index $l$. A power $f^{j}$ therefore lies in $\langle c\rangle$ if and only if $l\mid j$, and the same applies to $g$ with $m$ and to $\W$ with $n$.

        Let $a\in G$ be a torsion element, say with $a^{N}=1$ for some $N\ge1$. If its image in $T$ is trivial, then $a\in\langle c\rangle\cong\Z$ and hence $a=1$. Otherwise the image is a non-trivial torsion element of $T$. By Lemma \ref{lem:torsion-class} it is conjugate to a power $f^{j}$ with $0<j<l$, or to a similar power of $g$ or $\W$, which are treated in the same way. Lifting the conjugating element and replacing $a$ by its conjugate, we may assume that $a=f^{j}c^{t}$ for some $t\in\Z$, since $\langle c\rangle$ is the kernel of $G\to T$. Since $c$ is central, we have $1=a^{N}=f^{jN}c^{tN}$, and hence $f^{jN}$ lies in $\langle c\rangle$. This forces $l\mid jN$ and $f^{jN}=c^{jN/l}$. Now comparing exponents in the infinite cyclic group $\langle c\rangle$ we obtain $jN/l=-tN$, that is, $j=-tl$, which contradicts $0<j<l$. Hence $G$ is torsion-free.

        Finally, $\vcd(G)=3<\infty$ by Theorem \ref{thm:vcdG}, so $\cd(G)=\vcd(G)$ by Serre's theorem \cite[Ch.~VIII, Thm.~3.1]{brown}.
    \end{proof}
    
    \begin{corollary}\label{cor:FP}
        Suppose $k=2$ and $s\ge 2$. Then $G$ is of type $\FP$.
    \end{corollary}
    
    \begin{proof}
        Let $T_{0}$ and $M_{0}$ be as in Lemma \ref{lem:core} and let $G_{0}\le G$ be the preimage of $T_{0}$. As in Theorem \ref{thm:vcdG} there is a central extension $1\to\Z\to G_{0}\to T_{0}\to 1$, which is classified by some $e\in\Ho^{2}(T_{0},\Z)=\Ho^{2}(M_{0},\Z)$. Now let $E\to M_{0}$ be the principal $S^{1}$-bundle with Euler class $e$. Then since $\pi_{i}(M_{0})=0$ for $i\ge2$, the long exact homotopy sequence gives $\pi_{i}(E)=0$ for $i\ge2$ and implies that $\pi_{1}(S^{1})\to\pi_{1}(E)$ is injective. Hence $\pi_{1}(E)$ is a central extension of $T_{0}$ by $\Z$ with extension class $\pm e$, and thus $\pi_{1}(E)\cong G_{0}$. In particular, $E$ is a compact aspherical manifold and $G_{0}$ is of type $\FL$. Finally, $G_{0}\le G$ is a finite-index subgroup and, since $G$ is torsion-free by Proposition \ref{prop:G-torsionfree}, it follows that $G$ is of type $\FP$ \cite[Ch.~VIII, Prop.~6.6]{brown}.
    \end{proof}
    
   \section{Free resolutions}\label{sec:freeres}
    By Proposition \ref{prop:G-torsionfree} and Corollary \ref{cor:FP} the group $G$ is of type $\FP$ with $\cd(G)=3$, so the trivial module $\Z\in\ZGmod$ admits a projective resolution of length three. In this section we compute a \emph{free} resolution of length three explicitly, showing that $G$ is in fact of type $\FL$. This resolution is then used in \S\ref{sec:Tres} to produce an eventually $2$-periodic free resolution of $\Z$ over $\ZT$, by the construction of the first author \cite{seanyc}, and in Section \ref{sec:homology} to compute the integral homology and cohomology of both groups. Throughout Sections \ref{sec:freeres} and \ref{sec:homology} we assume $k=2$ and $s\ge 2$.
    
    We follow the notational conventions of \cite[\S2]{seanyc}. In particular, all modules are \emph{left} modules so that elements of a free module $\ZG^{p}$ are \emph{row} vectors and homomorphisms $\delta\colon\ZG^{p}\to\ZG^{q}$ are represented by a $p\times q$ matrix acting on the \emph{right}. Moreover, composition of homomorphisms is written left to right, so that $\delta d$ means ``first $\delta$, then $d$'', that is the usual matrix product $\delta\cdot d$. For $x\in G$ and $j\ge0$ we define the \emph{norm element} $\Ngk{x}{j}\coloneqq\sum_{i=0}^{j}x^{i}\in\ZG$, and abbreviate
    \[\Nf\coloneqq \Ngk{f}{l-1},\qquad \Ng\coloneqq \Ngk{g}{m-1},\qquad \Nw\coloneqq \Ngk{W}{n-1}.\]
    The augmentation map will always be denoted $\varepsilon\colon\ZG\to\Z$ (and $\ZT\to\Z$), and $\overline{(-)}$ will denote reduction along $\ZG\to\ZT=\ZG/(c-1)$. The same symbols $f,g,W,\Nf,\dots$ are used for the images of these elements in $\ZT$.
        
    Let $\mathcal{F}$ be the free group on the generators $f,g$ and $\om\colon\mathcal{F}\to G$ the canonical surjection. We write $D_{f},D_{g}\colon \mathcal{F}\to\ZG$ for the compositions of the \emph{Fox derivatives} $\mathcal{F}\to\Z\mathcal{F}$ with the map $\Z\mathcal{F}\to\ZG$ induced by $\om$. These satisfy $D_{x}(uv)=D_x(u)+\omega(u)D_{x}(v)$, $D_x(x^{-1})=-\om (x^{-1})$, and for any $y\in\mathcal{F}$ the \emph{fundamental formula} from \cite[\S2]{fox} is 
    \begin{equation}\label{Eq:FunForm}
        \omega(y)-1=D_{f}(y)(f-1)+D_{g}(y)(g-1).
    \end{equation}
    For the Farey word $W=\W$, we set
    \[\wf\coloneqq D_{f}(W),\qquad \wg\coloneqq D_{g}(W).\]
    
    We now recall the construction due to Lyndon which gives the first three terms of a free resolution for the trivial module over a finitely presented group \cite[\S5]{lyndon}.
    \begin{theorem}[Lyndon 1950]\label{Thm:preres}
    	Suppose $A=\langle x_{1},\ldots, x_{p} \mid  r_{1},\ldots, r_{q}\rangle$ is a finitely presented group. Then the trivial module $\Z\in\ZAmod$ has a free resolution beginning with
        \begin{equation*}\label{Eq:zres}
            \begin{tikzcd}
        		\Z A^{q} \arrow[r,"d_{2}"] & \Z A^{p} \arrow[r, "d_{1}"] & \Z A \arrow[r, "\varepsilon"] & \Z \arrow[r] & 0,
        	\end{tikzcd}
        \end{equation*}
        where $d_{1}\colon\Z A^{p}\to \Z A$ is the $p\times1$ matrix and $d_{2}\colon\Z A^{q}\to\Z A^{p}$ the $q\times p$ matrix
        \[d_{1}=\begin{bmatrix}
            x_{1}-1 \\
            \vdots \\
            x_{p}-1
        \end{bmatrix},\qquad d_{2} = \bigl[\,D_{x_{i}}(r_{j})\,\bigr]_{\substack{1\le j\le q\ (\text{rows})\\ 1\le i\le p\ (\text{columns})}} ,\]
        so that the rows of $d_{2}$ are indexed by the relators and its columns by the generators.
    \end{theorem}

    The following two elementary facts about group rings will be used repeatedly.

    \begin{lemma}\label{lem:regular}
        Let $A$ be a group and $h\in A$. If $h$ has infinite order, then $h-1$ is a non-zero-divisor in $\Z A$. If $h$ has finite order $d$, then the left annihilator $\Ann^{L}_{\Z A}(h-1)=\{a\in\Z A : a(h-1)=0\}$ is $\Z A\cdot\Ngk{h}{d-1}$.
    \end{lemma}
    
    \begin{proof}
        A set of representatives $X$ for the left cosets of $\langle h\rangle$ gives $\Z A=\bigoplus_{x\in X}x\,\Z\langle h\rangle$ as a right $\Z\langle h\rangle$-module. Moreover, right multiplication by $h-1$ preserves each summand, so both statements about products $a(h-1)$ reduce to the case $A=\langle h\rangle$. If $h$ has infinite order, then $\Z\langle h\rangle$ is the Laurent polynomial ring $\Z[h^{\pm1}]$, in which $h-1$ is a non-zero-divisor. If $h$ has finite order $d<\infty$, then $\Z\langle h\rangle$ is $\Z[\Z/d]$, in which the annihilator of $h-1$ is generated by $1+h+\dots+h^{d-1}$. Products $(h-1)a$ are treated in the same way using right cosets.
    \end{proof}
    
    For the central extension $G$, Lyndon's Theorem \ref{Thm:preres} supplies the first two differentials of the resolution. It then remains to show that $\ker(d_{2})$ is free of rank one, with an explicit generator.

    \subsection{Identities among relations} \label{subsec:pictures}
    For any finite presentation $\mathcal{P}$ of a group, the cellular chain groups of the universal cover of the presentation $2$-complex $X_{\mathcal{P}}$ are free on the cells and give the beginning of a free resolution for the trivial module, as in Theorem \ref{Thm:preres} (see \cite[\S1]{loday}). Since there are no $3$-cells, the Hurewicz theorem \cite[(20.1.1)]{tomdieck} gives
    \[
      \pi_{2}(X_{\mathcal{P}}) \cong \Ho_{2}(\widetilde{X}_{\mathcal{P}}) = \ker(d_{2}).
    \]
    The classes on the left are represented by \emph{spherical pictures}: discs labelled by relators, joined by arcs labelled by generators, on $S^{2}$ \cite[\S2]{howiepictures}. Rather than working with pictures, we read $\ker(d_{2})$ off the universal orbifold covering of $Q_{0}$. This computes the same group, and unlike the picture-theoretic methods recalled in Remark \ref{rem:whyk2} it needs no lower bound on the exponent $n$.

    Let $\mathcal{R}\trianglelefteq\mathcal{F}$ be the normal closure of the three relators $f^{l}, g^{m}, W^{n}=\W^{n}$ so that $T\cong\mathcal{F}/\mathcal{R}$. Consider the partial free resolution for $\Z\in\ZTmod$ obtained from Theorem \ref{Thm:preres}:
    \[
        \begin{tikzcd}
             \ZT^{3}  \arrow[r, "\delta_{2}"] & \ZT^{2}  \arrow[r, "\delta_{1}"] & \ZT \arrow[r, "\varepsilon"] & \Z \arrow[r] & 0,
        \end{tikzcd}
    \]
    where the differentials are given by
     \begin{equation}\label{eq:dT}
        \delta_{1}=\begin{bmatrix}
            f-1\\ 
            g-1
        \end{bmatrix}, \qquad \delta_{2}=\begin{bmatrix}
            \Nf & 0 \\ 
            0 & \Ng \\
            \Nw\wf & \Nw\wg
        \end{bmatrix}.
    \end{equation}
    The image $\im(\delta_{2})=\mathcal{R}^{\ab}$ is the relation module of the presentation of $T$, generated by the images of the three relators, and $\ker(\delta_{2})$ is the module of \emph{identities among relations}. Each relator is a proper power, which produces three identities
    \begin{equation*}\label{eq:zetas}
      \zeta_{1}=(f-1,0,0), \qquad \zeta_{2}=(0,g-1,0), \qquad \zeta_{3}=(0,0,W-1)
    \end{equation*}
    because $(f-1)\Nf=(g-1)\Ng=0$ and $(W-1)\Nw\wf=(W^{n}-1)\wf=0$ in $\ZT$, and similarly for $\wg$. The quotient of $\ZT^{3}$ by the submodule these three identities generate is $\ZT/\ZT(f-1)\oplus\ZT/\ZT(g-1)\oplus\ZT/\ZT(W-1)$. In Lemma \ref{hyp:AS} we use the geometry of the truncated orbifold $Q_{0}$, as set up in Section \ref{sec:orbifold}, to prove that these three identities generate the whole kernel $\ker(\delta_{2})$.
    
    Jones and Reid build $Q_{0}$ from a genus-two handlebody $H$, with $\pi_{1}(H)$ free on $f,g$, by attaching three \emph{singular $2$-handles} $V_{1},V_{2},V_{3}$. Each of these is a copy of $D^{2}\times I$ modulo a cyclic group of order $l$, $m$ or $n$ acting by rotation on the disc, so that its central arc $\{0\}\times I$ consists of cone points of that order. These handles are glued along disjoint annuli in $\partial H$ about curves representing $f$, $g$ and $\W$ respectively \cite[proof of Thm.~8]{jonesreid}. The Seifert--van Kampen theorem for orbifolds \cite[Cor.~2.3]{BoileauMaillotPorti2003}, applied once for each handle, presents $\pi_{1}^{\orb}(Q_{0})$ as $\pi_{1}(H)$ modulo the $l$-th, $m$-th and $n$-th powers of the attaching curves, which is precisely the presentation of $T$ used to define $\delta_{2}$. Since $k=2$, no neighbourhood of a vertex is added to the singular locus in this construction. Thus $Q_{0}=H\cup V_{1}\cup V_{2}\cup V_{3}$, where the handles are pairwise disjoint and each meets $H$ in its attaching annulus.
    
    The attaching curves are determined only up to free homotopy. In particular, the third handle represents $\W$ up to conjugacy, and it imposes a relator $\gamma W^{n}\gamma^{-1}$ for some $\gamma\in\pi_{1}(H)$. Since $W^{n}$ is trivial in $T$, the rules for Fox derivatives give $D_{x}(\gamma W^{n}\gamma^{-1})=\bar{\gamma}\,D_{x}(W^{n})$, so that if we replace $W^{n}$ by this relator, the third row of $\delta_{2}$ is multiplied on the left by the unit $\bar{\gamma}\in T$. Right multiplication by the invertible matrix $\operatorname{diag}(1,1,\bar{\gamma})$ maps the kernel of the new matrix onto $\ker(\delta_{2})$. Moreover, this matrix fixes $\zeta_{1}$ and $\zeta_{2}$, and sends the new third identity $(0,0,\bar{\gamma}W\bar{\gamma}^{-1}-1)$ to $\bar{\gamma}\zeta_{3}$. It follows that Lemma \ref{hyp:AS} holds for one presentation if and only if it holds for the other, and we may take the third curve to represent $\W$ without any loss of generality.
    
    \begin{lemma}\label{hyp:AS}
        The differential $\delta_{2}$ from \eqref{eq:dT} induces an isomorphism
        \[
            \begin{tikzcd}
                \ZT/\ZT(f-1)\ \oplus\ \ZT/\ZT(g-1)\ \oplus\ \ZT/\ZT(W-1) \arrow[r,"\sim"] & \mathcal{R}^{\ab}.
            \end{tikzcd}
        \]
        Equivalently, $\ker(\delta_{2})=\ZT\zeta_{1}+\ZT\zeta_{2}+\ZT\zeta_{3}$.
    \end{lemma}
    
    \begin{proof}
        Let $\pi|_{C}\colon C\to Q_{0}$ be the universal orbifold covering of Theorem \ref{thm:geometry}, so that $C$ is contractible with deck group $\Gamma\cong T$, and put $\widetilde{H}\coloneqq\pi^{-1}(H)$. Now since $\Ho_{2}(C)=\Ho_{1}(C)=0$, it follows from the long exact sequence of the pair $(C,\widetilde{H})$ that the connecting homomorphism
        \[
        \Delta\colon\Ho_{2}(C,\widetilde{H})\longrightarrow\Ho_{1}(\widetilde{H})
        \]
        is an isomorphism of $\ZT$-modules. We now identify $\Ho_{1}(\widetilde{H})$ and $\Ho_{2}(C,\widetilde{H})$, and then show that $\Delta$ is induced by $\delta_{2}$.
        
        Firstly, the map $\pi_{1}(H)\to T$ is surjective and, because $H$ avoids the singular locus, the covering $\pi|_{\widetilde{H}}\colon\widetilde{H}\to H$ is connected by \cite[Thm.~2.2]{BoileauMaillotPorti2003}. Moreover, the handlebody $H$ is a regular neighbourhood of a wedge of two circles representing $f$ and $g$, and it deformation retracts onto this wedge. By lifting that retraction over the covering \cite[(3.2.2) and (3.1.3)]{tomdieck}, we see that $\widetilde{H}$ deformation retracts onto the Cayley graph of $T$ with respect to $f$ and $g$. The associated cellular chain complex is given by $\ZT^{2}\to\ZT$ with boundary $\delta_{1}$, and since there are no $2$-cells we have
        \[
        \Ho_{1}(\widetilde{H})=\ker(\delta_{1})=\mathcal{R}^{\ab}.
        \]
        Here the second equality is the standard identification of the relation module, as in Theorem \ref{Thm:preres}. Under this identification, a relator lifts from the base point of $\widetilde{H}$ to a closed loop represented by the vector of its Fox derivatives \cite[Ch.~II, \S3]{hams}.
    
        To describe $\Ho_{2}(C,\widetilde{H})$, we consider the first handle $V_{1}$, which is the quotient of $D^{2}\times I$ by a rotation of order $l$ of the disc. The generator of $\pi_{1}^{\orb}(V_{1})\cong\Z/l$ maps to a conjugate of $f$, which has order $l$ in $T$ by Lemma \ref{lem:torsion-class}. It follows that each component of $\pi^{-1}(V_{1})$ is a copy of the universal cover $D^{2}\times I$ of $V_{1}$, with stabiliser given by a conjugate of $\langle f\rangle$ acting by rotations of the disc, and the group $T$ permutes these components transitively. Then for a suitable choice of base points, one of the components has stabiliser $\langle f\rangle$, and the boundary of a disc $D^{2}\times\{t\}$ in it is the lift of $f^{l}$ from the base point of $\widetilde{H}$. The same holds for $V_{2}$ and $V_{3}$, with $g^{m}$ and $W^{n}$ in place of $f^{l}$, respectively.
    
        For $i\in\{1,2,3\}$, each component of $\pi^{-1}(V_{i})$ meets $\widetilde{H}$ in its lateral boundary $\partial D^{2}\times I$, and distinct components are disjoint. Then since $Q_{0}=H\cup V_{1}\cup V_{2}\cup V_{3}$, excision identifies $\Ho_{2}(C,\widetilde{H})$ with the direct sum of the groups $\Ho_{2}(D^{2}\times I,\partial D^{2}\times I)\cong\Z$, one for each component, each of which is generated by the class of a disc $D^{2}\times\{t\}$. Moreover, the group $T$ acts by permuting these summands because it permutes the components. Furthermore, since rotations preserve the orientation of the disc, the stabiliser of a component acts trivially on its summand. The summands coming from $V_{1}$ therefore form the permutation module $\Z[T/\langle f\rangle]\cong\ZT/\ZT(f-1)$, generated by the class of such a disc with stabiliser $\langle f\rangle$. Treating $V_{2}$ and $V_{3}$ in the same way, we obtain an isomorphism
        \[
        \Ho_{2}(C,\widetilde{H}) \cong \frac{\ZT}{\ZT(f-1)}\oplus\frac{\ZT}{\ZT(g-1)} \oplus\frac{\ZT}{\ZT(W-1)}.
        \]
    
        Finally, $\Delta$ is $T$-equivariant and sends the class of a disc $D^{2}\times\{t\}$ to the class of its boundary circle. For the three generators above, these circles are the lifts of $f^{l}$, $g^{m}$ and $W^{n}$ from the base point, whose classes in $\Ho_{1}(\widetilde{H})$ are $(\Nf,0)$, $(0,\Ng)$ and $\Nw(\wf,\wg)$, where for the last we have used $D_{x}(W^{n})=\Nw D_{x}(W)$. These clearly coincide with the three rows of $\delta_{2}$ so that $\Delta$ is the desired isomorphism onto $\im(\delta_{2})\cong \mathcal{R}^{\ab}$.
    \end{proof}
    
    \begin{remark}\label{rem:whyk2}
        The hypothesis $k=2$ is used in the proof of Lemma \ref{hyp:AS} to decompose $Q_{0}$ into $H$ and three disjoint handles. Excision then splits $\Ho_{2}(C,\widetilde H)$ over the preimages of the three handles. For $k\le1$ the orbifold $Q_{0}$ also contains the cone on each spherical vertex link. These cones are glued to the ends of the handles, and thus excision no longer identifies $\Ho_{2}(C,\widetilde H)$ with the sum of the three induced modules.
    
        Howie's theory of pictures for one-relator products with a high-powered relator \cite{howiepictures,howiefourth} gives another route to Lemma \ref{hyp:AS} for part of our range, and it applies to arbitrary relators rather than only to Farey words. However, it needs a lower bound on the exponent of the relator. Indeed, the results of \cite{howiepictures} require $n\ge5$ and those of \cite{howiefourth} treat $n=4$ (cf.\ \cite[p.~588]{ellistri}), so the cases $n\in\{2,3\}$, such as the triple $(6,3,2)$ of Example \ref{ex:detzero}, lie outside their range.
    \end{remark}

    \subsection{The free resolution over \texorpdfstring{$\ZG$}{ZG}}
    With Lemma \ref{hyp:AS} in hand we can explicitly compute the resolution for the trivial module over $G=G_{W}(l,m,n)$.
    \begin{theorem}\label{thm:Gres}
        Suppose that $k=2$ and $s\geq 2$. Then the trivial module $\Z\in\ZGmod$ has a free resolution of the form
        \begin{equation}\label{Eq:GRes}
                \begin{tikzcd}
                	0 \arrow[r] & \ZG \arrow[r,"d_{3}"] & \ZG^{2} \arrow[r,"d_{2}"] & \ZG^{2} \arrow[r, "d_{1}"] & \ZG \arrow[r, "\varepsilon"] & \Z \arrow[r] & 0,
                \end{tikzcd}
        \end{equation}
        with differentials given by
        \[
            d_{1}=\begin{bmatrix}
                    f-1 \\
                    g-1
            \end{bmatrix},\;\; d_{2} = \begin{bmatrix}
                    \Nw\wf-\Nf & \Nw\wg  \\
                    \Nw\wf & \Nw\wg-\Ng  \\
                \end{bmatrix},\;\;d_{3} = \begin{bmatrix}
                \wf(f-1) & \wg(g-1)
            \end{bmatrix}.
        \]
        Thus $G$ is of type $\FL$ with $\cd(G)=3$.
    \end{theorem}
    \begin{proof}
        Exactness of the first three terms is Theorem \ref{Thm:preres}, applied to the relators $W^{n}f^{-l}$ and $W^{n}g^{-m}$. Next we check that $(\wf(f-1),\wg(g-1))$ lies in $\ker(d_{2})$. Using the fundamental formula \eqref{Eq:FunForm} and the fact that $c-1\in\ZG$ is central, we have
        \begin{align*}
            \wf(f-1)(\Nw\wf-\Nf) &+\wg(g-1)\Nw\wf \\
                &= (\wf(f-1)+\wg(g-1))\Nw\wf - \wf(f-1)\Nf \\
                &= (W-1)\Nw\wf-\wf(f^{l}-1) \\
                &= (W^{n}-1)\wf - \wf(f^{l}-1) \\
                &= (c-1)\wf - \wf(c-1) \\
                &=0,
        \end{align*}
        and similarly for the second coordinate. It remains to show that $\ker(d_{2})\subseteq\im(d_{3})$ and that $d_{3}$ is injective. We start with the former and let $(u_{1},u_{2})\in\ker(d_{2})$, so that
        \begin{equation*}\label{eq:d2ker}
            u_{1}(\Nw\wf-\Nf)+u_{2}\Nw\wf = 0, \quad u_{1}\Nw\wg + u_{2}(\Nw\wg-\Ng) = 0.
        \end{equation*}
        If we set $u\coloneqq u_{1}+u_{2}$, then these equations are equivalent to
        \begin{align*}
            u\Nw\wf = u_{1}\Nf, \quad u\Nw\wg = u_{2}\Ng.
        \end{align*}
        Multiplying these equations on the right by $f-1$ and $g-1$ respectively, we obtain
        \begin{align}
            u\Nw\wf(f-1) = u_{1}\Nf(f-1)=u_{1}(c-1), \label{eq:u1d2}\\ 
            u\Nw\wg(g-1) = u_{2}\Ng(g-1)=u_{2}(c-1). \label{eq:u2d2}
        \end{align}
        Reducing modulo the central element $c-1\in\ZG$, we have
        \begin{align*}
            \Bar{u}\Nw\wf(f-1) = \Bar{u}\Nw\wg(g-1) = 0\in\ZT.
        \end{align*}
        Since $f$ and $g$ have orders $l$ and $m$ in $T$ (Lemma \ref{lem:torsion-class}), Lemma \ref{lem:regular} gives $\Ann^{L}_{\ZT}(f-1)=\ZT\cdot\Nf$ and $\Ann^{L}_{\ZT}(g-1)=\ZT\cdot\Ng$, and hence
        \begin{align*}
            \Bar{u}\Nw\wf = v_{1}\Nf,\quad  \Bar{u}\Nw\wg=v_{2}\Ng,
        \end{align*}
        for some $v_{1},v_{2}\in \ZT$. Then $\zeta\coloneqq (-v_{1}, -v_{2}, \Bar{u})$ satisfies $\zeta\delta_{2}=0$, that is $\zeta\in \ker(\delta_{2})$. By Lemma \ref{hyp:AS} we can write $\zeta=a_{1}\zeta_{1}+a_{2}\zeta_{2}+a_{3}\zeta_{3}$ for some $a_{1},a_{2},a_{3}\in\ZT$, and by comparing third coordinates we obtain $\bar{u}=a_{3}(W-1)$.

        Lifting this equation back to $\ZG$ we obtain $u=\theta(W-1)+(c-1)\mu$, with $\theta,\mu\in\ZG$. Substituting this into \eqref{eq:u1d2} and \eqref{eq:u2d2} we have
        \begin{align*}
            u_{1}(c-1) = u\Nw\wf(f-1) &= (\theta(W-1)+(c-1)\mu)\Nw\wf(f-1)\\
            &= \theta(W-1)\Nw\wf(f-1) + (c-1)\mu\Nw\wf(f-1) \\
            &= (c-1)(\theta\wf(f-1) + \mu\Nw\wf(f-1)) \\
            &= (c-1)(\theta + \mu\Nw)\wf(f-1),\\
            \\
            u_{2}(c-1) = u\Nw\wg(g-1) &= (\theta(W-1)+(c-1)\mu)\Nw\wg(g-1) \\
            &=(c-1)(\theta + \mu\Nw)\wg(g-1).
        \end{align*}
        Since $c$ has infinite order in $G$ by Theorem \ref{thm:central}, it follows from Lemma \ref{lem:regular} that the central element $c-1$ is a non-zero-divisor. Cancelling $c-1$ from both sides above, we obtain
        \begin{align*}
            u_{1} = (\theta + \mu\Nw)\wf(f-1),\quad u_{2} = (\theta + \mu\Nw)\wg(g-1).
        \end{align*}
        That is, $(u_{1},u_{2}) = (\theta + \mu\Nw)d_{3}\in\im(d_{3})$, and hence $\ker(d_{2})=\im(d_{3})$ as required.

        Finally, we claim that $d_{3}\colon\ZG\to \ZG^{2}$ is injective. Indeed, if $\lambda d_{3}=0$ for some $\lambda\in\ZG$, then $\lambda\wf(f-1)=\lambda\wg(g-1)=0$, so that $\lambda(W-1)=0$ by the fundamental formula \eqref{Eq:FunForm}. Since $W^{n}=c$ has infinite order, so does $W$, and hence $W-1$ is a non-zero-divisor by Lemma \ref{lem:regular}, and $\lambda=0$.
    \end{proof}

    \subsection{The eventually \texorpdfstring{$2$}{2}-periodic resolution over \texorpdfstring{$\ZT$}{ZT}}\label{sec:Tres}
    We say that a chain complex $C_{\bullet}$ of free modules, with differentials $\partial_{j}\colon C_{j}\to C_{j-1}$, is \emph{$2$-periodic in degrees at least $k$} if $C_{j}=C_{j+2}$ for all $j\ge k$ and $\partial_{j}=\partial_{j+2}$ for all $j\ge k+1$, and \emph{eventually $2$-periodic} if this holds for some $k$. The two ranges differ by one because $\partial_{k}$ maps $C_{k}$ to $C_{k-1}$, which lies below the periodic range. If the modules $C_{j}$ with $j\ge k$ all have the same rank, we call this rank the \emph{stable rank} of $C_{\bullet}$.
    
    We obtain a resolution of $\Z\in\ZTmod$ by applying \cite[Thm.~9.7(2)]{seanyc} to the central extension $1\to\langle c\rangle\to G\to T\to 1$ and the resolution of Theorem \ref{thm:Gres}. The hypotheses of \cite[Thm.~9.7(2)]{seanyc} hold in our situation. Indeed, by Theorem \ref{thm:Gres} the group $G$ is of type $\FL$ with $\cd(G)=3$, and $c$ is a non-trivial central element of $G$ by Theorem \ref{thm:central}. As in \cite[\S9.2]{seanyc}, the element $c-1$ is then a central non-zero-divisor with $\ZG/(c-1)\cong\ZT$, and $\ZG$ has left global dimension at most four by \cite[Lem.~9.2]{seanyc}.
    
    \begin{theorem}\label{thm:Tres}
        Suppose that $k=2$ and $s\geq 2$. Then the trivial module $\Z$ admits a free resolution over
        $\ZT$
        \[
            \begin{tikzcd}
                \cdots \arrow[r, "\partial_{5}"] & \ZT^{3} \arrow[r, "\partial_{4}"] & \ZT^{3}  \arrow[r, "\partial_{3}"] & \ZT^{3}  \arrow[r, "\partial_{2}"] & \ZT^{2}  \arrow[r, "\partial_{1}"] & \ZT \arrow[r, "\varepsilon"] & \Z \arrow[r] & 0,
            \end{tikzcd}
         \]
        which is $2$-periodic in degrees at least two with stable rank three,
        where
        \begin{align*}
            \partial_{1}=\begin{bmatrix}
            f-1 \\
            g-1
            \end{bmatrix},\;\; \partial_{2}&=\begin{bmatrix}
                	    \Nf & 0\\
                        \Nw\wf-\Nf & \Nw\wg  \\
                        \Nw\wf & \Nw\wg-\Ng  \\
                	\end{bmatrix},\\ \partial_{3} =\begin{bmatrix}
                	    f-1 & 0 & 0\\
                        g-1 & g-1 & 1-g\\
                        0 & \wf(f-1) & \wg(g-1)
                	\end{bmatrix},\;\;
                    \partial_{4}&=\begin{bmatrix}
                	    \Nf & 0 & 0\\
                        \Nw\wf-\Nf & \Nw\wg & \Nw  \\
                        \Nw\wf & \Nw\wg-\Ng & \Nw  \\
                	\end{bmatrix},
        \end{align*}
        and $\partial_{i+2}=\partial_{i}$ for all $i\ge3$.
    \end{theorem}
    
    \begin{proof}
        Let $P_{\bullet}\to\Z$ be the resolution of Theorem \ref{thm:Gres} and put $z\coloneqq1-c$. The construction of \cite[Thms.~5.5 and 6.3]{seanyc} starts from a null-homotopy of multiplication by $z$ on $P_{\bullet}$. For $i=1,2,3$ we define $\varphi_{i}\colon P_{i-1}\to P_{i}$ by
        \begin{align*}
            \varphi_{1} = \begin{bmatrix}
            -\Nf & 0
            \end{bmatrix},\;\;\;
            \varphi_{2}=\begin{bmatrix}
            0 & 0 \\
            1-g & g-1
            \end{bmatrix}, \;\;\; \varphi_{3} = \begin{bmatrix}
            -\Nw \\
            -\Nw
            \end{bmatrix}.
        \end{align*}
        These are the dashed maps in the following diagram, whose vertical maps are multiplication by $z$.
        \[
            \begin{tikzcd}[ampersand replacement=\&, row sep=large, column sep=huge ]
                \ZG \& {\ZG^{2}} \& {\ZG^{2}} \& \ZG \\
                \ZG \& {\ZG^{2}} \& {\ZG^{2}} \& \ZG
                \arrow["{d_{3}}", from=1-1, to=1-2]
                \arrow["{z\cdot I}"', from=1-1, to=2-1]
                \arrow["{d_{2}}", from=1-2, to=1-3]
                \arrow["{\varphi_{3}}"', dashed, red, from=1-2, to=2-1]
                \arrow["{z\cdot I_{2}}"', from=1-2, to=2-2]
                \arrow["{\varphi_{2}}"', dashed, red, from=1-3, to=2-2]
                \arrow["{z\cdot I_{2}}"', from=1-3, to=2-3]
                \arrow["{{z\cdot I}}"', from=1-4, to=2-4]
                \arrow["{d_{1}}", from=1-3, to=1-4]
                \arrow["{\varphi_{1}}"', dashed, red, from=1-4, to=2-3]
                \arrow["{d_{3}}"', from=2-1, to=2-2]
                \arrow["{d_{2}}"', from=2-2, to=2-3]
                \arrow["{d_{1}}"', from=2-3, to=2-4]
            \end{tikzcd}
        \]
        A direct computation using the fundamental formula \eqref{Eq:FunForm} and $\Nf(f-1)=\Ng(g-1)=\Nw(W-1)=c-1$ gives
        \[
            \varphi_{1}d_{1} = z\cdot I, \quad d_{1}\varphi_{1}+\varphi_{2}d_{2} = z\cdot I_{2},  \quad d_{2}\varphi_{2}+\varphi_{3}d_{3} = z\cdot I_{2}, \quad d_{3}\varphi_{3} = z\cdot I.
        \]
        Thus the maps $\{\varphi_{i}\}_{i=1,2,3}$ form the required null-homotopy. Moreover, $\varphi_{1}\varphi_{2}=0$ and $\varphi_{2}\varphi_{3}=0$, so that Hypothesis~5.1 of \cite{seanyc} holds. Then we can apply \cite[Thm.~6.3]{seanyc} which produces the desired resolution with modules $\overline{P}_{0}$ and $\overline{P}_{1}$ in degrees zero and one, and $\overline{P}_{0}\oplus\overline{P}_{2}$ and $\overline{P}_{1}\oplus\overline{P}_{3}$ in all higher even and odd degrees respectively. Here $\overline{(-)}=\ZT\otimes_{\ZG}-$ denotes reduction modulo $z$. The resulting differentials are given in block form by
        \begin{align*}
            \partial_{1}=\od_{1},\;\;\partial_{2}=\begin{bmatrix}
            -\overline{\varphi}_{1} \\
            \od_{2}
            \end{bmatrix},\;\;
            \partial_{3}=\begin{bmatrix}
            \od_{1} & -\overline{\varphi}_{2} \\
            0 & \od_{3}
            \end{bmatrix},\;\;
            \partial_{4} = \begin{bmatrix}
            -\overline{\varphi}_{1} & 0\\
            \od_{2} & -\overline{\varphi}_{3}
            \end{bmatrix},
        \end{align*}
        and $\partial_{i+2}=\partial_{i}$ for $i\ge3$. Evaluating the blocks, we obtain the matrices in the statement.
    \end{proof}

    \begin{remark}\label{rem:comparison}
        For the slope $q=\tfrac12$, that is, for $W=W(1,2,f,g)=fg^{-1}f^{-1}g$, Ellis and Williams construct a free resolution of $\Z$ over $\ZT$ \cite[Thm.~3]{ellistri}. This resolution has ranks $6$, $13$, $13$ and $10$ in degrees zero to three, and it is $2$-periodic in degrees at least four with stable rank $9$. The resolution of Theorem~\ref{thm:Tres} has ranks $1$, $2$, $3$ and $3$ in degrees zero to three, and it is $2$-periodic in degrees at least two with stable rank $3$. Hence the rank is smaller in every degree, and the periodic range starts two degrees earlier.
    \end{remark}

    To compute the homology and cohomology of $T$ in Section~\ref{sec:homology}, we change bases so that the differentials of Theorem~\ref{thm:Tres} become diagonal in degrees at least three.

    \begin{corollary}\label{cor:diagonal}
        Suppose that $k=2$ and $s\ge 2$, let $\delta_{1}$ and $\delta_{2}$ be as in \eqref{eq:dT}, and put
        \[
        \delta_{3}\coloneqq\operatorname{diag}(f-1,\,g-1,\,W-1),\qquad
        \delta_{4}\coloneqq\operatorname{diag}(\Nf,\Ng,\Nw),\qquad
        \delta_{i+2}\coloneqq\delta_{i}\quad(i\ge 3),
        \]
        so that the rows of $\delta_{3}$ are the identities $\zeta_{1},\zeta_{2},\zeta_{3}$ of Lemma~\ref{hyp:AS}. Then
        \begin{equation}\label{eq:direct}
        \begin{tikzcd}
        \cdots \arrow[r, "\delta_{5}"] & \ZT^{3} \arrow[r, "\delta_{4}"] & \ZT^{3} \arrow[r, "\delta_{3}"] & \ZT^{3} \arrow[r, "\delta_{2}"] & \ZT^{2} \arrow[r, "\delta_{1}"] & \ZT \arrow[r, "\varepsilon"] & \Z \arrow[r] & 0
        \end{tikzcd}
        \end{equation}
        is a free resolution of $\Z$ over $\ZT$ which is $2$-periodic in degrees at least two. More precisely, let
        \[
        E\coloneqq\begin{bmatrix}1&0&0\\1&1&-1\\1&1&0\end{bmatrix},\qquad
        U\coloneqq\begin{bmatrix}1&0&0\\0&1&0\\\wf&\wg&1\end{bmatrix}.
        \]
        Then the identity in degrees zero and one, $E$ in every even degree at least two and $U$ in every odd degree at least three form a chain isomorphism from \eqref{eq:direct} to the resolution of Theorem~\ref{thm:Tres}.
    \end{corollary}
    
    \begin{proof}
        The matrix $E$ has determinant one and $U$ is unitriangular, so both are invertible over $\ZT$, with
        \[
            E^{-1}=\begin{bmatrix}
                1 & 0 & 0 \\
                -1 & 0 & 1 \\
                0 & -1 & 1
            \end{bmatrix},\qquad U^{-1}=\begin{bmatrix}
                1 & 0 & 0 \\
                0 & 1 & 0 \\
                -\wf & -\wg & 1
            \end{bmatrix}.
        \]
        Recall that matrices act on row vectors from the right. Since $\delta_{1}=\partial_{1}$ and both complexes repeat with period two from degree three onward, the maps in the statement form a chain map if and only if
        \[
            E\partial_{2}=\delta_{2},\qquad U\partial_{3}=\delta_{3}E,\qquad E\partial_{4}=\delta_{4}U.
        \]
        These identities follow directly by multiplying out the matrices of Theorem~\ref{thm:Tres}. The only entries that need more than cancellation lie in the third row of $U\partial_{3}$, which is
        \[
            \left(\wf(f-1)+\wg(g-1),\ \wf(f-1)+\wg(g-1),\ 0\right)=(W-1,\,W-1,\,0)
        \]
        by the fundamental formula \eqref{Eq:FunForm}, and this is the third row of $\delta_{3}E$. Hence the maps form a chain isomorphism, and \eqref{eq:direct} is a resolution because the complex of Theorem~\ref{thm:Tres} is one. 
        
        Finally, note that exactness of \eqref{eq:direct} can also be checked without Theorem~\ref{thm:Tres}. Indeed, it holds in degrees zero and one by Theorem~\ref{Thm:preres} and in degree two by Lemma~\ref{hyp:AS}. In degrees at least two, \eqref{eq:direct} is the direct sum of the standard $2$-periodic resolutions of $\Z$ over $\Z\langle f\rangle$, $\Z\langle g\rangle$ and $\Z\langle W\rangle$ (see Lemma~\ref{lem:torsion-class}), induced up to $\ZT$ and shifted by two, and these are exact in degrees at least three because $\ZT$ is free over each of these group rings.
    \end{proof}

    \begin{remark}\label{rem:cyclic}
        Since the free resolution \eqref{eq:direct} is induced from the cyclic subgroups $\langle f\rangle$, $\langle g\rangle$ and $\langle W\rangle$ in degrees at least two (see the end of the proof of Corollary~\ref{cor:diagonal}), the Eckmann--Shapiro lemma \cite[Thm. 4, p.488]{eckmannog} and the periodicity of the homology of finite cyclic groups give
        \[
            \Ho_{i}(T,A)\cong\Ho_{i}(\langle f\rangle,A)\oplus\Ho_{i}(\langle g\rangle,A)\oplus\Ho_{i}(\langle W\rangle,A)
        \]
        for every $\ZT$-module $A$ and every $i\ge 3$, and similarly for cohomology. For $A=\Z$ this explains the periodic part of Theorem~\ref{thm:homologyT}.
    \end{remark}

    \section{Integral homology and cohomology}\label{sec:homology}
    Both of the resolutions constructed in Section \ref{sec:freeres} consist of finitely generated free modules. If we apply $\Z\otimes_{\ZG}-$, respectively $\Z\otimes_{\ZT}-$, then every matrix entry is replaced by its augmentation, and the computation reduces to Smith normal forms over $\Z$. Throughout we use that, for an arbitrary group $A$, the universal coefficient theorem gives a non-canonical isomorphism
    \[
      \Ho^{i}(A,\Z)\simeq\Hom\bigl(\Ho_{i}(A,\Z),\Z\bigr)\oplus
      \Ext^{1}_{\Z}\bigl(\Ho_{i-1}(A,\Z),\Z\bigr).
    \]
    
    \subsection{The group \texorpdfstring{$G$}{G}}
    Applying the augmentation map to the resolution \eqref{Eq:GRes}, we obtain the complex of abelian groups
    \begin{equation*}
        \begin{tikzcd}
            \Z \arrow[r,"\varepsilon d_{3}"] & \Z^{2} \arrow[r,"\varepsilon d_{2}"] & \Z^{2} \arrow[r, " \varepsilon d_{1}"] & \Z,
        \end{tikzcd}
    \end{equation*} 
    with maps given by
    \begin{equation*}\label{eq:augG}
        \varepsilon d_{1}=\begin{bmatrix}
            0\\ 0
        \end{bmatrix},
        \qquad
        \varepsilon d_{2}=\begin{bmatrix}
            n\sigma_{f}-l & n\sigma_{g}\\
            n\sigma_{f} & n\sigma_{g}-m
        \end{bmatrix},
        \qquad
        \varepsilon d_{3}=\begin{bmatrix} 0 & 0 \end{bmatrix},
    \end{equation*}
    where we have used $\varepsilon(\wf)=\sigma_{f}$, $\varepsilon(\wg)=\sigma_{g}$, $\varepsilon(\Nf)=l$, $\varepsilon(\Ng)=m$ and $\varepsilon(\Nw)=n$.
    
    \begin{proof}[Proof of Theorem \ref{thm:homologyG}]
        Here $\varepsilon d_{2}$ is the matrix $J$ of Proposition \ref{prop:Gab}. The Smith normal form of $J$ is $\operatorname{diag}(\nu,|\det J|/\nu)$ if $\det J\ne 0$, and $\operatorname{diag}(\nu,0)$ otherwise. Since $\varepsilon d_{1}=0$, we have $\Ho_{1}(G,\Z)=\coker(\varepsilon d_{2})=\Gab$.
        
        Similarly, since $\varepsilon d_{3}=0$ we have $\Ho_{2}(G,\Z)=\ker(\varepsilon d_{2})$. For $\det J\ne0$ this vanishes. For $\det J=0$ the first row of $J$ is non-zero, so $J$ has rank one and the kernel is free of rank one. Moreover, because $\varepsilon d_{3}=0$ and the resolution has length three, we have $\Ho_{3}(G,\Z)=\Z$ and $\Ho_{i}(G,\Z)=0$ for $i\ge 4$. 
        
        The cohomology follows from the universal coefficient theorem. In particular, $\Ext^{1}_\Z(\Gab,\Z) \cong\Gab$ when $\Gab$ is finite, and $\Ext^{1}_\Z(\Z/\nu\oplus\Z,\Z) \cong\Z/\nu$ with $\Hom(\Z/\nu\oplus\Z,\Z)\cong\Z$ when $\det J=0$.
    \end{proof}
    
    \subsection{The group \texorpdfstring{$T$}{T}}
    By Corollary~\ref{cor:diagonal} we may compute with the free resolution \eqref{eq:direct}. Applying the augmentation map to its differentials, we obtain
    \[
        \varepsilon\delta_{1}=0,\qquad \varepsilon\delta_{2}=R_{T},\qquad \varepsilon\delta_{3}=0,\qquad \varepsilon\delta_{4}=\operatorname{diag}(l,m,n),
    \]
    where $R_{T}$ is the matrix of Proposition~\ref{prop:Tab}.
    
    \begin{proof}[Proof of Theorem~\ref{thm:homologyT}]
        Since $\varepsilon\delta_{1}=0$ and $\varepsilon\delta_{3}=0$, we have $\Ho_{1}(T,\Z)=\coker(R_{T})=\Tab$ by Proposition~\ref{prop:Tab}, and $\Ho_{2}(T,\Z)=\ker(R_{T})$. The latter is free of rank one, because the rows of $R_{T}$ span a lattice of rank two. In degrees $i\ge 3$ the matrix $\varepsilon\delta_{i}$ is zero for odd $i$ and equal to $\operatorname{diag}(l,m,n)$ for even $i$. It follows that $\Ho_{i}(T,\Z)\cong\Z/l\oplus\Z/m\oplus\Z/n$ for odd $i\ge 3$ and $\Ho_{i}(T,\Z)=0$ for even $i\ge 4$. The cohomology follows by applying the universal coefficient theorem to this homology. Here $\Tab$ is finite by Proposition~\ref{prop:Tab}, so that $\Hom(\Tab,\Z)=0$ and $\Ext^{1}_{\Z}(\Tab,\Z)\cong\Tab$.
    \end{proof}
    
    \begin{remark}\label{rem:optimal}
        When $\gcd(l,m,n)>1$, any free resolution of $\Z$ over $\ZT$ has rank at least three in each degree $i\ge3$. To see this, let $F_{\bullet}\to\Z$ be a free resolution with $F_{i}\cong\ZT^{t_{i}}$ for $i\ge 0$, and let $p$ be a prime dividing $l$, $m$ and $n$. Then the tensor product $\mathbb{F}_{p}\otimes_{\ZT} F_{\bullet}$ is a complex of $\mathbb{F}_{p}$-vector spaces of dimensions $t_{i}$ for $i\ge 0$. The homology of this complex is $\Ho_{*}(T,\mathbb{F}_{p})$, and hence $t_{i}\ge\dim\Ho_{i}(T,\mathbb{F}_{p})$. By Theorem \ref{thm:homologyT} and the universal coefficient theorem, $\Ho_{i}(T,\mathbb{F}_{p})$ is $(\Z/l\oplus\Z/m\oplus\Z/n)\otimes\mathbb{F}_{p}$ for odd $i\ge3$ and $\operatorname{Tor}^{\Z}_{1}(\Z/l\oplus\Z/m\oplus\Z/n,\mathbb{F}_{p})$ for even $i\ge4$. In particular, both of these vector spaces are three-dimensional because $p$ divides $l$, $m$ and $n$. Since the resolution of Theorem \ref{thm:Tres} has rank three in all degrees $i\ge 2$, it is optimal in this sense when $\gcd(l,m,n)>1$.
        
        When $\gcd(l,m,n)=1$ the same argument gives a lower bound of at most two, and we do not know whether stable rank two can be attained in this case.
    \end{remark}

    \subsection{Two examples}\label{sec:examples}
    The differentials of Theorems \ref{thm:Gres} and \ref{thm:Tres} depend on the Farey word only through its Fox derivatives $\wf$ and $\wg$, which can be read off \eqref{eq:Wrs} directly. We work out one example for each parity of $s$. The first falls under case (a) of Proposition \ref{prop:Gab} and the second under case (c).

    \begin{example}[$q=\tfrac12$, the family of Ellis and Williams]\label{ex:half}
        For $(r,s)=(1,2)$ the Farey word is the commutator $W=fg^{-1}f^{-1}g$, and its Fox derivatives are
        \[
            \wf=1-fg^{-1}f^{-1}, \qquad \wg=fg^{-1}f^{-1}-fg^{-1}.
        \]
        We take $(l,m,n)=(3,3,4)$. Then both vertex links are the hyperbolic turnover $S^{2}(3,3,4)$, with $c_{1}=c_{2}=\tfrac{2}{3}+\tfrac{1}{4}-1=-\tfrac{1}{12}< 0$, so that the hypothesis $k=2$ holds. Since $s$ is even, Proposition \ref{prop:Gab}(a) gives $\det J=lm=9$ and $\bar c=0$. The abelianisations of $G$ and $T$ therefore agree, and both are isomorphic to $\Z/3\oplus\Z/3$. Theorems \ref{thm:homologyG} and \ref{thm:homologyT} now give
        \[
            \Ho_{i}(G,\Z)\cong\begin{cases}
                \Z & i=0,\\
                \Z/3\oplus \Z/3 & i=1,\\
                0 & i=2, \\
                \Z & i=3,\\
                0 & i\ge 4,
            \end{cases}
            \qquad
            \Ho_{i}(T,\Z)\cong\Z/3\oplus\Z/3\oplus\Z/4 \quad(i\ge3\ \text{odd}).
        \]
        Here $\Ho_{2}(G,\Z)$ vanishes because $\det J\ne0$, while $\Ho_{3}(G,\Z)\cong\Z$ comes from the last term of the resolution of Theorem \ref{thm:Gres}, whose differential vanishes after applying the augmentation map. In odd degrees the homology of $T$ is the sum of the homology of its cyclic subgroups $\langle f\rangle$, $\langle g\rangle$ and $\langle W\rangle$, which carry all of its torsion by Lemma \ref{lem:torsion-class}. This is the family treated by Ellis and Williams \cite[\S5]{ellistri}, and Remark \ref{rem:comparison} compares the resolution constructed in \cite{ellistri} with ours.
    \end{example}

    \begin{example}[$q=\tfrac23$, the case $\det J=0$]\label{ex:detzero}
        For $(r,s)=(2,3)$ the Farey word is $W=fg^{-1}fgf^{-1}g$, and its Fox derivatives are
        \[
            \wf=1+fg^{-1}-fg^{-1}fgf^{-1}, \qquad \wg=fg^{-1}f+fg^{-1}fgf^{-1}-fg^{-1}.
        \]
        In this case we take $(l,m,n)=(6,3,2)$. Then $c_{1}=c_{2}=\tfrac16+\tfrac13+\tfrac12-1=0$, and both vertex links are the Euclidean turnover $S^{2}(6,3,2)$, which is one of the three turnovers of Lemma \ref{lem:euclidean-links}. Here $\eta=3$ and $\nu=1$, and
        \[
          \det J = lm-n(l\sigma_{g} + m\sigma_{f}) = 18-2\cdot 9 = 0 .
        \]
        This example therefore falls under case (c) of Proposition \ref{prop:Gab}, which by Lemma \ref{lem:detJk2} can only occur under the present hypotheses.
    
        Since $\nu=1$, the abelianisation $\Gab\cong\Z$ is torsion-free. The class $\bar c$ generates a subgroup of index $\lcm(l,m)=6$, and hence $\Tab=\Gab/\langle\bar c\rangle\cong\Z/6$. The homomorphism $\Psi$ of Lemma \ref{lem:Psi} sends $f$, $g$ and $c$ to $1$, $2$ and $6$ respectively, and it generates $\Ho^{1}(G,\Z)\cong\Z$. Theorem \ref{thm:homologyG} now gives $\Ho_{i}(G,\Z)\cong\Z$ for $0\le i\le3$ and $\Ho_{i}(G,\Z)=0$ for $i\ge4$. In particular $\Ho_{2}(G,\Z)$ is non-zero, and by Theorem \ref{thm:central}(b) the map $\alpha$ identifies it with $\Ho_{2}(T,\Z)\cong\Z$. The situation is reversed in Example \ref{ex:half}, where $\Ho_{2}(G,\Z)=0$ and $\xi$ maps $\Ho_{2}(T,\Z)$ isomorphically onto $\langle c\rangle$. Finally, Theorem \ref{thm:homologyT} gives $\Ho_{i}(T,\Z)\cong\Z/6\oplus\Z/3\oplus\Z/2$ for odd $i\ge3$.
    \end{example}

    \printbibliography

\end{document}